%% file: AffineDeformations.tex
\documentclass[12pt]{article}

\usepackage{amsmath}
\usepackage{amsthm}
\usepackage{amssymb}
\usepackage{caption}
\usepackage{xspace}
\usepackage{latexsym}
\usepackage[curve]{xypic} 
\usepackage{cite}
\usepackage{enumitem}
\usepackage{mathtools}
\usepackage{color}
\usepackage{tikz}
\usepackage{tikz-cd}
\usepackage{hyperref} 
\usepackage{verbatim}
\usepackage{xcolor}
\usepackage{caption}
\usepackage{subcaption}
\usepackage[all]{xy}
\usepackage{graphicx}
\usepackage{xparse}

\numberwithin{equation}{section}

\theoremstyle{plain}
\newtheorem{thm}{Theorem}[section]
\newtheorem{cor}[thm]{Corollary}

\newtheorem{prop}[thm]{Proposition}
\newtheorem{lemma}[thm]{Lemma}
\newtheorem{lem}[thm]{Lemma}

\newtheorem{theorem}[thm]{Theorem}
\theoremstyle{definition}

\newtheorem{definition}[thm]{Definition}
\theoremstyle{remark}

\newtheorem*{remark}{Remark}

\newcommand{\nc}{\newcommand}
\newcommand{\rnc}{\renewcommand}

\newcommand{\isom}{\cong}

\newcommand{\proj}{\mathbb{P}}

\newcommand{\Teich}[1][g,n]{\mathcal{T}_{#1}}

\newcommand{\moduli}[1][g,n]{\mathcal{M}_{#1}}
\newcommand{\barmoduli}[1][g,n]{\overline{\mathcal{M}}_{#1}}
\newcommand{\Dev}{\mathrm{Dev}}
\newcommand{\Hol}{\mathrm{Hol}}
\newcommand{\THol}{\widetilde{\Hol}}

\newcommand{\dR}{\mathrm{dR}}

\newcommand{\TA}{\widetilde{\mathcal{A}}_{g,n}}
\newcommand{\TAres}{\widetilde{\mathcal{A}}'_{g,n}}
\newcommand{\ModA}{{\mathcal{A}}_{g,n}}
\newcommand{\ModAres}{{{\mathcal{A}'}_{g,n}}}
\newcommand{\barModA}{{\overline{\mathcal{A}}}_{g,n}}

\newcommand{\Cframe}[1][g,n]{F_{#1}}

\newcommand{\Def}{\mathrm{Def}_1}

\NewDocumentCommand{\deriv}{O{{}} m m}{\frac{d^#1 #2}{d#3^#1}}
\newcommand{\tnabla}{\widetilde{\nabla}}
\newcommand{\tbnabla}{\widetilde{\overline{\nabla}}}

\newcommand{\ra}{\longrightarrow}

\newcommand{\xra}{\mathop{\longrightarrow}^}

\newcommand{\ds}{\displaystyle}

\newcommand{\pderiv}[2]{\frac{\partial #1}{\partial #2}}

\nc{\dmo}{\DeclareMathOperator}
\rnc{\Re}{\operatorname{Re}}
\rnc{\Im}{\operatorname{Im}}
\dmo{\coker}{coker}
\dmo{\rank}{rank}
\dmo{\End}{End}
\dmo{\Hom}{Hom}
\dmo{\Jac}{Jac}
\dmo{\Id}{Id}
\dmo{\Ann}{Ann}
\dmo{\Area}{Area}
\dmo{\CP}{\bC P^1}
\dmo{\Aut}{Aut}
\dmo{\Aff}{Aff}
\dmo{\Sp}{\mathrm{Sp}}

\nc\bA{\mathbb{A}}
\nc\bB{\mathbb{B}}
\nc\bC{\mathbb{C}}
\nc\bD{\mathbb{D}}
\nc\bE{\mathbb{E}}
\nc\bF{\mathbb{F}}
\nc\bG{\mathbb{G}}
\nc\bH{\mathbb{H}}
\nc\bI{\mathbb{I}}
\nc{\bJ}{\mathbb{J}} 
\nc\bK{\mathbb{K}}
\nc\bL{\mathbb{L}}
\nc\bM{\mathbb{M}}
\nc\bN{\mathbb{N}}
\nc\bO{\mathbb{O}}
\nc\bP{\mathbb{P}}
\nc\bQ{\mathbb{Q}}
\nc\bR{\mathbb{R}}
\nc\bS{\mathbb{S}}
\nc\bT{\mathbb{T}}
\nc\bU{\mathbb{U}}
\nc\bV{\mathbb{V}}
\nc\bW{\mathbb{W}}
\nc\bY{\mathbb{Y}}
\nc\bX{\mathbb{X}}
\nc\bZ{\mathbb{Z}}
\nc\cA{\mathcal{A}}
\nc\cB{\mathcal{B}}
\nc\cC{\mathcal{C}}
\rnc\cD{\mathcal{D}}
\nc\cE{\mathcal{E}}
\nc\cF{\mathcal{F}}
\nc\cG{\mathcal{G}}
\rnc\cH{\mathcal{H}}
\nc\cI{\mathcal{I}}
\nc{\cJ}{\mathcal{J}} 
\nc\cK{\mathcal{K}}
\rnc\cL{\mathcal{L}}
\nc\cM{\mathcal{M}}
\nc\cN{\mathcal{N}}
\nc\cO{\mathcal{O}}
\nc\cP{\mathcal{P}}
\nc\cQ{\mathcal{Q}}
\rnc\cR{\mathcal{R}}
\nc\cS{\mathcal{S}}
\nc\cT{\mathcal{T}}
\nc\cU{\mathcal{U}}
\nc\cV{\mathcal{V}}
\nc\cW{\mathcal{W}}
\nc\cY{\mathcal{Y}}
\nc\cX{\mathcal{X}}
\nc\cZ{\mathcal{Z}}
\nc\bfA{\mathbf{A}}
\nc\bfB{\mathbf{B}}
\nc\bfC{\mathbf{C}}
\nc\bfD{\mathbf{D}}
\nc\bfE{\mathbf{E}}
\nc\bfF{\mathbf{F}}
\nc\bfG{\mathbf{G}}
\nc\bfH{\mathbf{H}}
\nc\bfI{\mathbf{I}}
\nc{\bfJ}{\mathbf{J}} 
\nc\bfK{\mathbf{K}}
\nc\bfL{\mathbf{L}}
\nc\bfM{\mathbf{M}}
\nc\bfN{\mathbf{N}}
\nc\bfO{\mathbf{O}}
\nc\bfP{\mathbf{P}}
\nc\bfQ{\mathbf{Q}}
\nc\bfR{\mathbf{R}}
\nc\bfr{\boldsymbol{r}}
\nc\bfS{\mathbf{S}}
\nc\bfT{\mathbf{T}}
\nc\bfU{\mathbf{U}}
\nc\bfV{\mathbf{V}}
\nc\bfW{\mathbf{W}}
\nc\bfY{\mathbf{Y}}
\nc\bfX{\mathbf{X}}
\nc\bfZ{\mathbf{Z}}

\nc\bfmu{\boldsymbol{\mu}}
\nc{\ModX}{\mathrm{Mod}(X)}
\nc{\QuadX}{\mathrm{Quad}(X)}
\nc{\Ztwo}{\mathbb{Z}/2\mathbb{Z}}

\rnc{\Re}{\operatorname{Re}}
\rnc{\Im}{\operatorname{Im}}
\dmo{\kernel}{ker}
\dmo{\cokernel}{coker}
\dmo{\image}{im}
\dmo{\rk}{rk}
\dmo{\rel}{rel}
\dmo{\Twist}{\mathrm{Twist}}
\dmo{\TwistX}{\mathrm{Twist}(X, \omega)}
\dmo{\Pic}{Pic}
\dmo{\Res}{Res}
\dmo{\Spec}{Spec}

\DeclareMathOperator{\ord}{ord}

\newcommand{\cx}{\mathbb{C}}
\newcommand{\reals}{\mathbb{R}}

\newcommand{\zed}{\mathbb{Z}}
\newcommand{\bfm}{\boldsymbol{m}}

\newcommand{\trans}[1][X]{\mathfrak{trans}_{#1}}

\renewcommand{\tilde}{\widetilde}

\newcommand{\Teichmuller}{Teichm\"uller\xspace}

\title{Holonomy of affine surfaces}

\author{Paul Apisa, Matt Bainbridge, Jane Wang}

\date{}

\begin{document}

\maketitle

\begin{abstract}
  We identify the moduli space of complex affine surfaces with the
  moduli space of regular meromorphic connections on Riemann surfaces
  and show that it satisfies a corresponding universal property. As a
  consequence, we identify the tangent space of the moduli space of
  affine surfaces, at an affine surface $X$, with the first
  hypercohomology of a two-term sequences of sheaves on $X$.  In terms
  of this identification, we calculate the derivative and coderivative
  of the holonomy map, sending an affine surface to its holonomy
  character.  Using these formulas, we show that the
  holonomy map is a submersion at every affine surface that is not a
  finite-area translation surface, extending work of Veech
  \cite{veech93}. Finally, we introduce a holomorphic foliation of
  some strata of meromorphic affine surfaces, which we call the
  isoresidual foliation, along whose leave holonomy characters and
  certain residues are constant.  We equip this foliation with a
  leafwise indefinite Hermitian metric, again extending work of Veech.
\end{abstract}

\input{sec_intro.tex}

\input{sec_affine_surfaces.tex}

\input{sec_families.tex}

\input{sec_deformations.tex}

\input{sec_holonomy_new.tex}

\input{sec_stuff2}

\input{sec_stuff.tex}

\bibliography{my}{}
\bibliographystyle{halpha}
\end{document}

%% file: sec_intro.tex
\section{Introduction}

Given a Riemann surface $X$, an \emph{affine structure} on $X$ is an
atlas of charts, compatible with the complex structure, whose
transition functions belong to the group $\Aff(\cx)$ of complex affine
automorphisms of the complex plane.  A \emph{branched affine
  structure} is an affine structure on $X \setminus C$, where
$C\subset X$ is a discrete set of \emph{cone points}, subject to a
certain moderate growth condition on the developing map near cone
points (see \S\ref{sec:affine} for a precise definition).
The \emph{order} of a cone point $c$ is $\tau(\gamma)-1$, where
$\tau(\gamma)$ is the complex turning number of a positively oriented
loop around $c$, generalizing the usual notion of the order of a cone
point of a surface with flat metric.  We will call a cone point a
\emph{zero} if the order of its real part is greater than $-1$ and a \emph{pole}
otherwise.  

A \emph{(branched) affine surface} is a Riemann surface
with a (branched) affine structure\footnote{ In terms of projective structures on Riemann surfaces, our ``branched affine structures" are ``projective structures with regular singularities", which is strictly more general than ``branched projective structures".}.  Since nearly all of the affine
structures we consider are branched, we will usually omit the
adjective ``branched'' in the sequel.

There is a well-known correspondence\footnote{See Rakhimov \cite{Rakhimov} and Novikov-Shapiro-Tahar \cite{NovikovShapiroTahar}. See also Abate-Bianchi \cite{AbateBianchi-MeroConnRS} and Abate-Tovena \cite{Abate-Tovena-VectorField} for a connection to dynamics of polynomial vector fields. } between branched affine structures
on $X$ and holomorphic connections 
\begin{equation*}
  \nabla\colon \Omega_X\to \Omega_X\otimes \Omega_X(C)
\end{equation*}
on the holomorphic cotangent bundle $\Omega_X$ with regular
singularities along $C$ (meaning the connection form may have at worst
simple poles along $C$).   Briefly, this correspondence can be
understood as follows.  The standard connection on the complex plane
is $\Aff(\cx)$-invariant, and so its pullbacks to $X\setminus C$ by
the affine charts glue to a connection over $X \setminus C$.
Conversely, the solutions to the second-order differential equation $\nabla df=0$ define
the atlas of charts for an affine structure on $X\setminus C$. Under
this correspondence, the order of a cone point $c$ is the residue
$\Res_c\nabla = \Res_c\nabla\omega$, for $\omega$ a nonzero holomorphic
one-form near $c$.  We recall the details of this correspondence in
\S\ref{sec:affine}. Following this correspondence, we will
generally consider an affine surface to be a triple $(X, C, \nabla)$,
where $X$ is a closed Riemann surface, $C\subset X$ is a finite set,
and $\nabla$ is a connection on $\Omega_X$ with regular singularities
along $C$.

Affine surfaces notably generalize translation surfaces, surfaces with
an atlas of charts whose transition functions are translations of the
complex plane, or equivalently, Riemann surfaces equipped with a
nonzero meromorphic one-form.  From the connection point of view, an
affine surface is a translation surface if it has a flat (meaning in
the kernel of $\nabla$) meromorphic
one-form.

\paragraph{The moduli space.}

Let $\ModA$ denote the moduli space of genus $g$ affine surfaces with
$n$ cone points, and let $\ModA(\mu)\subset \ModA$ denote the stratum
whose cone points have orders specified by a tuple $\mu$ of complex
numbers summing to $2g-2$.  Likewise, let $\TA$ and $\TA(\mu)$ denote
the corresponding \Teichmuller spaces of marked affine surfaces.  Following ideas of Veech
\cite{veech93}, we constructed $\cA_{g,n}$ as a fiber product in
Apisa-Bainbridge-Wang \cite{ABW} and showed that it is an affine bundle
over the moduli space of pointed curves $\cM_{g,n}$ modeled on the
extended Hodge bundle of meromorphic one-forms with at worst simple
poles at the punctures.

The connection point of view readily generalizes to define the notion
of an affine structure on a stable curve, or more generally on a
family of stable curves.  In \S\ref{sec_families}, we define an
affine structure on a family of stable curves as a relative connection
on the relative dualizing sheaf.  We then construct the moduli space
$\barModA$ of affine stable curves, obtain its universal property, and
show that it extends the affine bundle $\ModA\to \moduli$ to an affine
bundle over the Deligne-Mumford compactification $\barmoduli$:

\begin{thm}\label{T:Families}
  There is an fine moduli space $\barModA$ of affine stable curves
  which is an affine bundle over $\barmoduli$.
\end{thm}

While $\barModA$ is not compact, as the fibers of this affine bundle are
not compact, there is a natural smooth compactification of
$\ModA$ obtained by adjoining to $\barModA$ the bundle of projective
spaces at infinity over $\barmoduli$.

\paragraph{The tangent bundle.}

By the universal property of $\ModA$, each tangent space
$T_{(X, C, \nabla)}\ModA$ is naturally isomorphic to the space of
first-order deformations of $(X, C, \nabla)$.  In \cite[Proposition 4.5]{HubbardMasur},
Hubbard and Masur calculated the tangent bundle to the bundle of quadratic differentials
over the moduli space of Riemann surfaces by identifying the
first-order deformations of a pair $(X, q)$ with the first
hypercohomology of the two-term complex $T_X\to \Omega_X^{2}$, where the
differential sends a vector field $v$ to the Lie derivative
$L_v(q)$ (see also \cite{moller_linear} and \cite{BCGGM2} for similar
computations for strata of $k$-differentials).  In \S\ref{sec:deformations}, we give a similar
calculation of the first-order deformations of an affine surface.
Given an affine surface, consider the two-term complex of sheaves
$L^\bullet$,
\begin{equation}
  \label{eq:2term2complex}
  \xymatrix{
    T_X(-C) \ar[r]^{\mathcal{L}_\nabla} & \Omega_X(C)
  },
\end{equation}
where again $\mathcal{L}_\nabla(v)$ is the Lie derivative of $\nabla$
with respect to $v$.

\begin{thm}
  \label{T:DerivativeOfHol}
  There is a natural identification of the tangent space of $\ModA$
  at $(X, C, \nabla)$ with the hypercohomology $\bH^1(L^\bullet)$.
\end{thm}

\paragraph{Holonomy and turning numbers.}

Given an affine surface $(X, C, \nabla)$ and a closed curve
$\gamma\colon[0,1]\to X\setminus C$, the \emph{holonomy} of $\nabla$
around $\gamma$ is $\chi(\gamma)= v(1)/v(0)\in \cx^*$, where $v$ is a
flat vector field along $\gamma$.  This defines a holonomy character
$\chi\in H^1(X\setminus C, \cx^*)$.  Passing to the \Teichmuller space
$\TA$ of affine surfaces $(X, C, \nabla)$ with a marking
$(\Sigma, S)\to (X, C)$, we obtain a holonomy map
$\Hol\colon \TA\to H^1(\Sigma\setminus S, \cx^*)$.

There is a notion of the \emph{turning number} $\tau(\gamma)$
of a closed immersed curve $\gamma$ in $X\setminus C$, generalizing
the standard notion of the turning number of a curve in a surface with
a flat metric, which is related to the holonomy by
\begin{equation*}
  e^{2\pi i \tau(\gamma)} = \chi(\gamma).
\end{equation*}
These turning numbers can be encoded as a class in the unit circle
bundle $T^1(X\setminus C)$.  The homotopy exact sequence of $T^1(\Sigma\setminus C)\to \Sigma\setminus C$
induces the short exact sequence
\begin{equation*}
  0 \to H^1(\Sigma \setminus S, \cx) \to H^1(T^1(\Sigma\setminus S),
  \cx) \to H^1(S^1, \cx) \to 0.
\end{equation*}
Let $\Cframe\subset H^1(T^1(\Sigma \setminus S), \cx)$ denote the
preimage of the fundamental class of the fiber $S^1$, which we call
the space of \emph{complex framings} of $(\Sigma, S)$.  The exact
sequence shows that $\Cframe$ is an affine space modeled on
$H^1(\Sigma \setminus S, \cx)$, and moreover there is a covering map
$\mathrm{exp}_*\colon \Cframe\to H^1(\Sigma\setminus S, \cx^*)$
induced by exponentiation of coefficients.  An affine surface has an
associated framing which assigns the turning number $\tau(\gamma)$ to
the canonical lift $\tilde{\gamma}$ of an immersed curve $\gamma$ to
$T^1(X\setminus C)$.  Recording the framings of each affine surface
defines a map $\THol\colon \TA\to \Cframe$ which is a lift of $\Hol$
in the sense that $\exp_*\circ \THol = \Hol$.

The terminology ``framing'' refers to a result
of Johnson \cite{johnson} establishing an equivalence between
integral framings and trivializations of $T(\Sigma \setminus S)$ up to homotopy.

In \S\ref{sec:derivativeofholonomy}, we calculate the derivative
of $\THol$. Since $\Cframe$ is an affine space over
$H^1(\Sigma\setminus S, \cx)$, its tangent spaces are naturally
identified with
$H^1(X\setminus C, \cx)\isom \bH^1(\Omega_X^\bullet(\log C))$ by the
Algebraic de Rham Theorem, where $\Omega_X^\bullet(\log C)$ is the
logarithmic de Rham complex.  In terms of this and the identification
of $T_{(X,C, \nabla)}\TA$ with $\bH^1(L^\bullet)$, we have:

\begin{theorem}
  \label{thm:holonomy_derivative_intro}
  The derivative $D\THol$ at the point representing an affine surface
  $(X, C, \nabla)$ is $\frac{1}{2\pi i}$ times the map
  $\mathbb{H}^1(L^\bullet)\to \mathbb{H}^1(\Omega_X(\log C))$ induced
  by the following morphism of complexes.
  \begin{equation*}
    \xymatrix{
      \Omega_X(C) \ar[r]^-{-\mathrm{id}} & \Omega_X(C) \\
      T_X(-C) \ar[r]^-{\nabla^*}\ar[u]^{\mathcal{L}_\nabla}  & \mathcal{O}_X \ar[u]^d
    }
  \end{equation*}
\end{theorem}
The above horizontal arrow $\nabla^*\colon T_X(-C) \to \mathcal{O}_X$ is
shorthand for the composition of the dual connection $\nabla^*\colon
T_X(-C)\to T_X(-C)\otimes \Omega_X(C)$ with the contraction
$T_X(-C)\otimes \Omega_X(C)\to \mathcal{O}_X$.  The commutivity of this diagram follows from Proposition~\ref{prop:LieDerivative} below.

In \S\ref{sec:submersion}, we derive from
Theorem~\ref{thm:holonomy_derivative_intro} a simple formula for
the coderivative of $\THol$:

\begin{theorem}
  \label{T:coderivative}
  The coderivative $D^*\THol$ at an affine surface $(X, C, \nabla)$ fits in the diagram,
  \begin{equation*}
    \xymatrix{
      0 \ar[r] & H^0(\Omega_X) \ar[r]\ar[d]^{-\frac{1}{2\pi i}\nabla} & H_1(X \setminus C, \cx) \ar[r]\ar[d]^{D^*\THol} & H^1(\mathcal{O}_X(-C))\ar[r]\ar[d]^{-\frac{1}{2\pi i} \mathrm{id}} & 0\\
      0 \ar[r] & H^0(\Omega_X^2(C)) \ar[r] & T_{(X,C,\nabla)}^*\TA \ar[r] & H^1(\mathcal{O}_X(-C))\ar[r] & 0}
  \end{equation*}
\end{theorem}

The top row of this diagram is the Hodge filtration, and the bottom row arises from the affine bundle $\ModA\to\moduli$.

\paragraph{The isoholonomic foliation.}

Restricting to affine surfaces without poles, Veech \cite{veech93}
showed that $\Hol$ is a submersion away from the locus $\Hol^{-1}(1)$ of affine
surfaces with trivial holonomy (that is, translation surfaces).  As an
application of the coderivative formula, we
generalize Veech's result to arbitrary affine surfaces.

\begin{thm}\label{T:Veechsubmersion}
  $\Hol$ is a submersion at an affine surface $X$ if and only if $X$
  is not a finite-area translation surface.
\end{thm}

\begin{proof}
  By the coderivative formula, surjectivity of $D\THol$ at an affine
  surface $(X,C, \nabla)$ is equivalent to the injectivity of
  $\nabla\colon \Omega_X\to \Omega_X^2(C)$.  In other words, $D\THol$
  fails to be surjective exactly when there is a flat holomorphic
  one-form on $X$.
\end{proof}

 Theorem \ref{T:Veechsubmersion} also follows from work on
the monodromy of projective structures, as we discuss below.

On the level of the
moduli space $\ModA$, it follows that the fibers of $\Hol$ descend to
an \emph{isoholonomic foliation} of the complement of the locus of
finite area translation surfaces.

\paragraph{The isoresidual foliation.}

For strata $\ModA(\mu)$ with at least one pole of integral order, we
introduce in \S\ref{sec:isoresidual} a subfoliation of the
isoholonomic foliation, the \emph{isoresidual foliation}, along whose
leaves the residues of flat one-forms defined near integral poles are 
projectively constant. We will see that this defines a
holomorphic foliation of the locus, $\ModAres(\mu)\subset\ModA(\mu)$
(or $\TAres(\mu)\subset\TA(\mu)$) of surfaces which are not
translation surfaces and have some integral pole of nonzero residue.

Given an affine surface $(X, C, \nabla)$ in a stratum $\ModA(\mu)$
with at least one integral pole, let
$P_{\zed}=\{p_1, \ldots, p_k\} \subset C$ denote the set of integral
poles.    Choose a tree
$\Gamma=\{\gamma_2, \ldots, \gamma_k\}$, where $\gamma_j$ is an
embedded arc from $p_1$ to $p_j$.  Let $\omega_\Gamma$ be a flat one-form
defined on a neighborhood of $\Gamma$.  Since $\omega_\Gamma$ is only
defined up to scale, the tuple of residues of $\omega_\Gamma$
naturally belongs to the projectivization
$\proj^{P_\zed}=\proj(\cx^{P_\zed})$.  Drawing the tree $\Gamma$
instead on the marking surface $\Sigma$, this defines a holomorphic
function $\Res_\Gamma\colon \TAres(\mu)\to\proj^{P_\zed}$.

\begin{thm}\label{T:SubmersionWithResidue}
  The map
  $\Hol\times \Res_\Gamma\colon \TAres(\mu)\to H^1(\Sigma\setminus
  S;\cx^*)\times \proj^{P_\zed}$ is a submersion whose fibers are a
  well-defined foliation of 
  $\TAres(\mu)$ which does not depend on the choice of $\Gamma$ and
  which descends to a holomorphic foliation of $\ModAres(\mu)$.  
\end{thm}

Theorem~\ref{T:SubmersionWithResidue} is proved by analyzing a
spectral sequence, some of whose differentials represent the
derivatives $D\THol$ and $D\Res_\Gamma$.  Analyzing the same spectral
sequence yields a second proof of Theorem~\ref{T:Veechsubmersion}.

We define in \S\ref{sec:isoresidual} an indefinite Hermitian metric on
the leaves of this isoresidual foliation restricted to the locus
$\mathcal{F}'_{g,n}(\mu)\subset \ModAres(\mu)$ of flat surfaces, generalizing a metric on the
leaves of the isoholonomic foliation, constructed by Veech
\cite{veech93} in the holomorphic case.

\paragraph{Connections to previous work.}\label{SS:PreviousWork}


Branched affine structures are examples of \emph{projective structures (with
  regular singularities)} on surfaces. Subject to conditions on the
singularities, which we will not describe, such a structure is a
Riemann surface $X$ with a finite collection of points $C$ and an
atlas of charts to $\bP^1$ with transition functions in
$G = \mathrm{PGL}(2, \bC)$. These structures have a moduli space that
is naturally a complex variety and a monodromy map to a character
variety $\mathrm{Hom}(\pi_1(X \setminus C), G)/G$.

In work of Hubbard \cite{Hubbard-ProjectiveStructures}, Earle \cite{Earle-ProjStructures}, and Hejhal \cite{Hejhal}, the moduli space of affine and projective structures with no singularities is considered and the monodromy map is shown to be a local homeomorphism. Hubbard \cite[Proposition 2]{Hubbard-ProjectiveStructures} identifies the tangent space to the moduli space of unbranched projective structures with the cohomology of a rank three local system. To do this, he builds a universal family of unbranched projective structures \cite[Proposition 1]{Hubbard-ProjectiveStructures} and appeals to Kodaira-Spencer theory. He also computes the derivative of the monodromy map as an induced map between cohomology groups. This work is closely connected to our own approach to proving Theorems \ref{T:Families}, \ref{T:DerivativeOfHol}, and \ref{thm:holonomy_derivative_intro}.

Earle's approach is more analytic and rooted in classical Teichm\"uller theory. He also considers projective structures with no singularities, but observes that the Teichm\"uller space of such structures can be identified with an affine space over Teichm\"uller space whose fiber is modeled on the space of quadratic differentials. He identifies the tangent space to the character variety with an Eichler cohomology group and shows that its derivative coincides with a specific map from $Q(X) \times B(X)$ to this cohomology group where $Q(X)$ and $B(X)$ are spaces of quadratic differentials and Beltrami differentials respectively on a Riemann surface $X$. This is called the \emph{Earle variational formula}.

While Earle's approach is not obviously related to our work, Luo \cite{Luo-Monodromy} used the Earle variational formula to show that the monodromy map is a local diffeomorphism when restricted to projective structures with regular singularities, none of which have parabolic local monodromy. Iwasaki \cite{Iwasaki-Monodromy} then extended the analysis to the setting where a certain number of cone points are \emph{apparent singularities}, i.e. have trivial local monodromy. The work of Iwasaki and Luo is essentially strong enough to contain Theorem \ref{T:Veechsubmersion} as a corollary. A special case of Iwasaki's result appears in work of Billon \cite{Billon-ModuliOfBranchedProj}.

The problem of determining when the monodromy map is a local biholomorphism to the character variety, phrased in the language of $\mathrm{PSL}(2, \bC)$-opers instead of branched projective structures, was studied by Allegretti-Bridgeland \cite{AllegrettiBridgeland}, who permit projective structures with irregular singularities. Serandour \cite{serandour2023meromorphic} solved this problem with a result that subsumes and generalizes all previously discussed work in this subsection. His work directly implies Theorem \ref{T:Veechsubmersion}.

\paragraph{Acknowledgments.} This work was supported by a grant from
the Simons Foundation [\#713192, MB].
The first author was partially supported by
NSF grant no.\ DMS-2304840.  


%% file: sec_affine_surfaces.tex
\section{Affine Surfaces}
\label{sec:affine}

In this section, we recall in more detail the correspondence between
branched complex affine structures on a Riemann surface $X$ and
connections with regular singularities on the holomorphic cotangent
bundle $\Omega_X$ (see also \cite{NovikovShapiroTahar}).  The equivalence
between these notions is mediated by the classical Riemann-Hilbert
correspondence which identifies vector bundles over $X$ equipped with
a flat connection with representations $\pi_1(X)\to\mathrm{GL}_n\cx$,
of which we recall the special case that we need and refer the reader
to \cite{conrad_riemann_hilbert} and
\cite{deligne_equations_differentielles} for the general case.  As an
illustration of the connection point of view, we recover Veech's
classification of affine structures on a pointed Riemann surface.

\paragraph{Flat bundles.}

Given a Riemann surface $X$ together with a holomorphic line bundle
$\mathcal{L}\to X$, recall that a \emph{(holomorphic) connection} $\nabla$ on
$\mathcal{L}$ is a $\cx$-linear homomorphism of sheaves,
\begin{equation*}
  \nabla\colon \mathcal{L}\to \Omega_X\otimes_{\mathcal{O}_X} \mathcal{L},
\end{equation*}
which satisfies for any holomorphic function $f$ and local section
$\sigma$ the Leibniz rule,
\begin{equation*}
  \nabla(f\sigma) = df \otimes \sigma + f \nabla\sigma.
\end{equation*}
A section $s$ of $\mathcal{L}$ is \emph{flat} if $\nabla s=0$.  We
call a pair $(\mathcal{L}, \nabla)$ a \emph{flat line bundle} over
$X$.  One would ordinarily want to assume that the curvature of
$\nabla$ is $0$, but this is automatic if $X$ is one-dimensional.

It follows from the Leibniz rule that the difference of two
connections on $\mathcal{L}$ is $\mathcal{O}_X$-linear, and so a
holomorphic section of $\Omega_X\otimes \End(\mathcal{L})= \Omega_X$.
In other words, the set of connections on $\mathcal{L}$ is an affine
space over $\Omega(X) = H^0(\Omega_X)$.

A local trivialization $s$ of $\cL$ determines a
holomorphic one-form $\Gamma$, the \emph{connection form}, such that
\begin{equation*}
  \nabla s = \Gamma \otimes s.
\end{equation*}
If $s' = g s$ is a different trivialization, its connection form $\Gamma'$ is
related to $\Gamma$ by the transformation,
\begin{equation}
  \label{eq:1}
  \Gamma' = dg/g + \Gamma.
\end{equation}

A connection $\nabla$ on $\mathcal{L}$ induces a dual connection
$\nabla^*$ on $\mathcal{L}^*$, the dual of $\mathcal{L}$, which is
characterized by the Leibniz rule,
\begin{equation*}
  \langle  \nabla s, t^*\rangle + \langle s, \nabla^* t^*\rangle= d
  \langle s, t^*\rangle
\end{equation*}
where $s$ and $t^*$ are local sections of $\mathcal{L}$ and
$\mathcal{L}^*$, respectively.

The following is the classical Riemann-Hilbert correspondence for the
special case of one-dimensional line bundles over Riemann surfaces.
\begin{theorem}[Riemann-Hilbert correspondence]
  \label{thm:RiemannHilbert}
  Given a Riemann surface $X$, the following structures are
  equivalent:
  \begin{enumerate}
  \item a flat line bundle $(\mathcal{L}, \nabla)$ over $X$.
  \item a one-dimensional complex local system $V$ over $X$.
  \item a \emph{holonomy character}  $\chi\in H^1(X, \cx^*)$.
  \end{enumerate}
\end{theorem}

\begin{proof}[Sketch of Proof]
  The bundle of flat sections $K = \ker\nabla$ is a local system by the
  Picard–Lindel\"of Theorem.  The character $\chi$ assigns to a loop
  $\gamma$ the holonomy of a flat section over $\gamma$.  Given a
  character $\chi$, we may define a flat bundle with holonomy $\chi$
  as $\mathcal{L}_\chi = \widetilde{X}\times \cx / \pi_1(X)$, where
  $\pi_1(X)$ acts on the universal cover $\widetilde{X}$ by deck
  transformations, and on $\cx$ by the character $\chi$.  The trivial
  connection $\nabla=d$ is invariant under this action, and so it
  descends to a connection $\nabla_\chi$ on $\mathcal{L}_\chi$.
\end{proof}

Now suppose $(\mathcal{L}, \nabla)$ is a flat line bundle over
$X\setminus C$, a Riemann surface with a discrete set of punctures.  A
section $\sigma$ of $\mathcal{L}$ is \emph{meromorphic} if it is
meromorphic over $X\setminus C$ in the usual sense and moreover every
puncture $c$ has a neighborhood in which $\sigma = z^m f(z) s$, where
$s$ is a multivalued flat section, $z$ is a local coordinate sending
$c$ to $0$, and $f$ is meromorphic.  The \emph{order} of $\sigma$ at
$c$ is the complex number
\begin{equation}
  \label{eq:meromorphic_order}
  \ord_c \sigma = m+ \ord_0 f.
\end{equation}

\paragraph{Regular singularities.}

Now suppose $C\subset X$ is a discrete set of punctures and $\mathcal{L}$ is a line
bundle over $X'=X\setminus C$.  A connection $\nabla$ on
$\mathcal{L}|_{X\setminus C}$ is said to have \emph{regular
  singularities along $C$} if its connection forms are meromorphic
with at worst simple poles along $C$.  In other words,
$\nabla$ is a $\cx$-linear homomorphism
\begin{equation*}
  \nabla\colon \mathcal L \to \Omega_X(C)\otimes_{\mathcal{O}_X} \mathcal{L}.
\end{equation*}
The \emph{residue} $\Res_c \nabla$ of $\nabla$ at $c$ is defined to be
the residue of the connection form $\Gamma$ for any local section,
which is well defined by \eqref{eq:1}.

It is well-known that the regularity of a connection is equivalent to
bounded growth of multivalued sections at the singularities.  More
precisely, a multivalued function $f$ on the punctured disk has
\emph{moderate growth} if for any sector $S$ and single-valued branch
of $f$ on $S$, we have $|f(z)| \leq C |z|^{-N}$ on $S$ for some
constants $C$ and $N>0$.  A multivalued section $s$ of
$\mathcal{L}\to X\setminus C$ has \emph{moderate growth} if in a
neighborhood of each puncture, we have $s=f s_0$, where $s_0$ is a
single-valued holomorphic
section and $f$ is a multivalued function of moderate growth.

\begin{theorem}[{\cite[Theorem~1.19]{deligne_equations_differentielles}}]
  \label{thm:regulariffmoderate}
  A connection $\nabla$ on $\mathcal{L}|_{X\setminus C}$ has regular
  singularities along $C$ if and only if every multivalued flat section
  has moderate growth at every puncture.
\end{theorem}

The residue of a connection at a puncture determines the holonomy of
flat sections around that puncture:

\begin{prop}
  If $c$ is a regular singularity of a flat bundle $(\mathcal{L},
  \nabla)$ of residue $r$, then in a trivialization over a  neighborhood of $c$, any flat section is of the form $z^{-r} F(z)$, where $z$ is a local
  coordinate identifying $c$ with $0$, and $F$ is holomorphic
  near $0$.  In particular, $s$ has holonomy $e^{-2\pi i r}$ around $c$.
\end{prop}

\begin{proof}
  Let $s$ be a trivialization of $\mathcal{L}$ around $c$ with
  connection form $\Gamma = (\frac{r}{z} + f(z))dz$.  A flat section
  is of the form $g s$, where $g$ is a multivalued
  holomorphic function satisfying the differential equation
  $dg + g\Gamma=0$, so $g(z) = z^{-r}e^{\int f(z)\,dz}$.  
\end{proof}

\begin{theorem}
  \label{thm:bundle_extension}
  Given a Riemann surface $X$ and discrete subset $C$, any flat bundle
  $(\mathcal{L}', \nabla')$ over $X'=X \setminus C$ may be extended to
  a flat bundle $(\mathcal{L}, \nabla)$ over $X$ with regular
  singularities along $C$ whose residue at each puncture $c$ may be taken
  to be any $m$ satisfying
  \begin{equation}
    \label{eq:2}
    e^{2 \pi i m} = a,
  \end{equation}
  where $a$ is the holonomy of $(\mathcal{L}', \nabla')$ around
  $c$.  Moreover, any two such extensions with the same residues are
  isomorphic as flat singular bundles.
\end{theorem}

\begin{remark}
  This is a special case of the general principle that extending a
  flat vector bundle across a singularity amounts to a choice of
  logarithm of the monodromy (see
  \cite[Proposition~5.4]{deligne_equations_differentielles}).
\end{remark}

\begin{proof}
  Choose local coordinates identifying a neighborhood of $c$ with
  the punctured unit disk $\Delta^*$.  Let $s$ be a multivalued flat
  section of $\mathcal{L}'|_{\Delta^*}$.  For any $m$
  satisfying \eqref{eq:2}, $t = z^{m}s$ defines a trivialization of
  $\mathcal{L}'|_{\Delta^*}$, which defines an extension of
  $\mathcal{L}'$ across $c$ in which $t$ is a nonzero holomorphic
  section.  We compute,
  \begin{equation*}
    \nabla' t = m z^{m-1} dz\otimes s = m \frac{dz}{z}\otimes t,
  \end{equation*}
  so $\nabla'$ extends to a connection with a regular singularity at
  $c$ of residue $m$.

  Given two extensions $(\mathcal{L}_i. \nabla_i)$, the tensor product
  $(\mathcal{L}_1\otimes\mathcal{L}_2^*, \nabla_1\otimes \nabla_2^*)$
  has zero residue at each puncture, so it is a nonsingular flat
  bundle over $X$ with trivial holonomy.  It follows that
  $\mathcal{L}_1\otimes\mathcal{L}_2^*$  is trivial, so
  $\mathcal{L}_1\isom\mathcal{L}_2$ as flat bundles.
\end{proof}

Given a holonomy character $\chi\in H^1(X\setminus C, \cx^*)$ together
with an assignment $\bfm\colon C \to \cx$ of residues satisfying
\eqref{eq:2} at each puncture we call the pair $(\chi, \bfm)$ the
\emph{holonomy data} for $X\setminus C$.  By
Theorems~\ref{thm:RiemannHilbert} and~\ref{thm:bundle_extension},
holonomy data $(\chi, \bfm)$ determines a unique flat bundle
$(\mathcal{L}_{\chi, \bfm}, \nabla_{\chi, \bfm})$.  If $X$ is compact,
the product of the monodromy around the punctures is trivial, so the
sum of residues $d = \sum_{c\in C} \bfm(c)$ is integral.  By \cite{ABW}, $d$ is
the degree of $\mathcal{L}_{\chi, \bfm}$.

A meromorphic section of a flat bundle with regular singularities has
two competing notions of order at a singularity, which are related by
the following proposition.

\begin{prop}
  \label{prop:residueisorder}
  Given a meromorphic section $\sigma$ of a flat bundle
  $(\mathcal{L}, \nabla)$ over a neighborhood of a regular
  singularity $c$, the order $\ord_c\sigma$ as defined in
  (\ref{eq:meromorphic_order}) and the order
  $\ord_c^\mathcal{L}\sigma$ of $\sigma$ as a section of
  $\mathcal{L}$ are related by
  \begin{equation*}
    \ord_c \sigma = \Res_c\nabla + \ord_c^\mathcal{L}\sigma,
  \end{equation*}
\end{prop}

\paragraph{Affine surfaces.}

Recall that a branched affine structure on a Riemann surface $X$ with cone
points along $C$ is an atlas of charts covering $X\setminus C$ with
transition functions in $\Aff(C)$, subject to the following moderate
growth condition near each cone point.  A \emph{developing map} is a
multivalued holomorphic function on $X\setminus C$ whose
precomposition with any affine chart is affine.  We require that any
developing map has moderate growth near each cone point.

The cotangent bundle of the complex plane has its standard
$\Aff(\cx)$-in\-var\-i\-ant connection $\nabla(f\, dz) = f' \, dz^2$.  Given
an affine surface $X$, pulling back this standard connection by 
affine charts covering $X\setminus C$ defines a holomorphic connection $\nabla$ on the
cotangent bundle of $X\setminus C$.  

A closed curve $\gamma\colon [0,1] \to X\setminus C$ on an affine
surface $X$ has a complex turning number $\tau(\gamma)$,
generalizing the notion of turning number on a flat surface.   We define
the \emph{turning number of $\gamma$ with respect to $\nabla$} to be
\begin{equation}
  \label{eq:turning_number}
  \tau(\gamma) = \frac{1}{2\pi i} \int_0^1
  \frac{D(\gamma'(t))}{\gamma'(t)} \, dt,
\end{equation}
where $D$ is the covariant derivative associated to $\nabla^*$,
assigning to a vector field $v$ along $\gamma$ the contraction of
$\nabla^*v$ with $\gamma'$.
In other terms, suppose
$v(t) = g(t) \gamma'(t)$ is a parallel vector
field along $\gamma$, meaning $Dv=0$.  Then $g'/g =
-D\gamma'/\gamma'$, so
$\tau(\gamma) = \frac{1}{2\pi i}(\log g(0) - \log g(1))$.

We define the \emph{order} of a cone point $c$ to be
$\tau(\gamma)-1$ and the \emph{cone angle} to be
$2\pi \tau(\gamma)$, where $\gamma$ is a positively oriented loop around $c$.
The cone angle of $c$ is real when there is a compatible flat metric
in a neighborhood of $c$, in which case it agrees with the usual
notion of cone angle for metric cone points.  When the cone
angle is not real, its imaginary part represents the dilation of a
compatible flat metric around $c$.

The following theorem shows the equivalence of several competing
definitions of affine surfaces which have appeared in the literature.
Definition (a) is close to definitions appearing in \cite{veech93} and
\cite{gunning78}.  Definition (b) appears in \cite{Wang}.
We introduced definition (c) in \cite{ABW} and
established its equivalence with definition (b).  Definitions close to
(d) appear in \cite{gunning_special} and \cite{kra_affine}. The equivalence of definitions (a) and (d) was shown in \cite{NovikovShapiroTahar}. 

\begin{theorem}\label{T:DefinitionOfAffine}
  Given a Riemann surface $X$ with discrete subset $C$ and a function
  $\bfm\colon C \to \cx$, the following structures are equivalent:
  \begin{enumerate}[label=(\alph*)]
  \item an affine structure on $X$ having order $-\bfm(c)$ at each
    $c\in C$.
  \item a character $\chi\in H^1(X\setminus C, \cx^*)$ together with a
    meromorphic section $\omega$ of $\mathcal{L}_\chi\otimes \Omega_X$
    (defined up to scalar multiples) which is holomorphic and nonzero
    over $X\setminus C$ and has order $\bfm(c)$ at each $c\in C$
  \item holonomy data $(\chi, \bfm)$ such that the line bundle
    $\mathcal{L}_{\chi, \bfm}\otimes \Omega_X$ is holomorphically
    trivial.
  \item a connection $\nabla$ on $\Omega_X$ with regular singularities
    along $C$ having residues $-\bfm(c)$ at each $c\in C$.
  
  \end{enumerate}
\end{theorem}

\begin{proof}
  By Proposition~\ref{prop:residueisorder}, a meromorphic section
  $\omega$ as in (b) is equivalent to a global trivialization of
  $\mathcal{L}_{\chi, \bfm}\otimes \Omega_X$ as in (c).  Since two global
  trivializations differ by scalar multiple, $\omega$ is taken to be
  defined up to scalar multiple in (b).

  To show that (c) implies (d) we note that a trivialization of
  $\mathcal{L}_{\chi, \bfm}\otimes \Omega_X$ given by (c) defines an
  isomorphism $\Omega_X \to \mathcal{L}_{\chi, \bfm}^*$.  Give
  $\Omega_X$ the dual connection $\nabla_{\chi, \bfm}^*$.  This
  connection does not depend on the choice of trivialization, as
  constant scalar multiplication commutes with any connection.

  Conversely,  to show that (d) implies (c), given a connection on $\Omega_X$ with regular
  singularities, let $(\chi, \bfm)$ be its holonomy data.  Then
  $\Omega_X\isom \mathcal{L}_{\chi, \bfm}$, so
  $ \mathcal{L}_{\chi, \bfm}^*\otimes\Omega_X \isom
  \mathcal{L}_{\chi^{-1}, -\bfm}\otimes\Omega_X$ is trivial.

  For the equivalence between (a) and (d), we've already seen that an
  affine structure defines a holomorphic connection on the cotangent
  bundle of $X\setminus C$.  Conversely, let $\{U_\alpha\}$ be a
  covering of $X\setminus C$ by simply connected open sets.  Over
  $U_\alpha$, let $\omega_\alpha$ be a flat holomorphic one-form, and
  define a chart $\phi_\alpha(z) = \int_p^z \omega_\alpha$ for some
  base point $p\in U_\alpha$.  These charts identify $\nabla$ with the
  standard connection on $\cx$, and it follows that the transition
  functions are affine.

  To see that moderate growth of a developing map $\Dev$ is equivalent to regular
  singularities of $\nabla$, consider the multivalued flat one-form $\omega =
  \Dev^*dz$ over $X\setminus C$.   In local coordinates around a
  puncture, the derivative $\Dev'$ has moderate growth if and only if
  $\Dev$ has moderate growth by the Cauchy Integral Formula (see
  \cite[Theorem~9.10]{sauloy}).  It follows that moderate growth of $\Dev$ is equivalent
  to moderate growth of $\omega$, which is equivalent to regular
  singularities of $\nabla$ by Theorem~\ref{thm:regulariffmoderate}.

  It remains to relate the order of a cone point $c$ to the residue of
  $\nabla$ at $c$.  We apply Proposition~\ref{prop:turning difference}
  below, where $\nabla_1=\nabla$ and $\nabla_2$ is the trivial
  connection, defined in a neighborhood of $c$ by a choice of local
  coordinates.  Since a positively oriented loop $\gamma$ around $c$ has
  turning number $1$ with respect to $\nabla_2$, the turning number of
  $\gamma$ with respect to $\nabla_1$ is $\tau(\gamma) = 1 - \Res_c
  \nabla$.  It follows that the order of $c$ is $-\Res_c\nabla$.
\end{proof}

In the sequel, we adopt the point of view of Definition~(d):
an \emph{affine surface} is a triple $(X, C,\nabla)$, where $X$ is a
Riemann surface, $C\subset X$ a discrete subset, and $\nabla$ is a connection on $\Omega_X$ with
regular singularities along $C$.  We will refer to $\nabla$ as an
\emph{affine structure on $(X, C)$}.

The Gauss-Bonnet Theorem for cone metrics generalizes in the obvious
way to affine surfaces (for the cone metric case in arbitrary
dimensions,
see McMullen \cite{McMullen-GaussBonnet} and the references therein).

\begin{prop}
  For any compact affine surface $(X, \nabla)$, the orders of the cone
  points sum to $-\chi(X)$.
\end{prop}

\begin{proof}
  Consider $\delta = \log|\chi|\in H^1(X\setminus C, \reals)$, where $\chi$
  is the holonomy character, and let $\omega$ be the unique
  meromorphic one-form with at worst simple poles along $C$ and $[\Re
  \omega]=\delta$.  The connection $\nabla'=\nabla-\omega$ has $S^1$-holonomy,
  so it is compatible with a conformal metric with the same cone
  points and the same sum of cone angles.  The claim then follows from
  the usual Gauss-Bonnet Theorem.
\end{proof}

Recall that a cone point is a \emph{zero} if the real part of its
order greater than $-1$; otherwise it is a \emph{pole}.  Write
$C = P \cup Z$, where $Z$ and $P$ are the sets of zeros and poles. In
the sequel, the set $C$ of cone points will always be finite. We
enumerate them as $C = \{c_1, \ldots, c_n\}$ and write $m_j$ for the
order of $c_j$.

Veech \cite{veech93} showed that a compact genus $g$ Riemann surface $X$ has an affine
structure with a given set of cone points and orders if and only if
the orders sum to $2g-2$.  Moreover, he showed that the set of such
affine structures is an affine space over $\Omega(X)$. (A special case
of this appears in \cite{gunning_special}.)  This is
readily seen from the connection point of view via the Atiyah class,
which is the obstruction to a holomorphic vector bundle admitting a
connection (see \cite{atiyah_connections, huybrechts}).

\begin{theorem}
  Given a Riemann surface $X$ with a discrete subset $C\subset X$ and
  a function $m\colon C\to\cx$, there is an affine structure on $X$
  with cone points along $C$ and orders given by $m$ if and only if
  the values of $m$ sum to $-\chi(X)$ or if $X$ is not compact.  The set
  of affine structures with these cone points and orders is an affine
  space over $H^0(\Omega_X)$ and the set of all affine structures with
  these cone points is an affine space over $H^0(\Omega_X(C))$.
\end{theorem}

\begin{proof}
  Choose an open cover $\mathcal{U} = \{U_\alpha\}$ of $X$ together
  with nowhere-zero forms $\omega_\alpha\in \Omega_{X}(U_\alpha)$.
  Let $g_{\alpha\beta} = \omega_\alpha/\omega_\beta$, a nonzero
  holomorphic function over $U_{\alpha\beta} = U_\alpha\cap U_\beta$.
  A connection over $U_\alpha$ is defined by
  $\nabla \omega_\alpha = \eta_\alpha \otimes \omega_\alpha$ for some
  choice of meromorphic forms
  $\eta_\alpha \in \Omega_{X}(C)(U_\alpha)$.  We calculate
  \begin{align*}    
    &\nabla \omega_\alpha = \eta_\alpha\otimes\omega_\alpha=
    g_{\alpha\beta} \eta_\alpha \otimes\omega_\beta\\
   & \nabla g_{\alpha\beta}\omega_\beta = dg_{\alpha\beta}\otimes\omega_\beta + g_{\alpha\beta}\eta_\beta\otimes\omega_\beta,
  \end{align*}
  so these glue to a global connection when
  \begin{equation}
    \label{eq:atiyah_class}
    \eta_\alpha - \eta_\beta = \frac{dg_{\alpha\beta}}{g_{\alpha\beta}}.
  \end{equation}
  The cocycle on the right-hand-side defines the Atiyah class $A \in
  H^1(\Omega_{X}(C))$, and we see that the desired connection exists
  if and only if $A=0$.

  For any $X$ with $C\neq \emptyset$, we have $H^1(\Omega_X(C))=0$, so
  the desired connection exists.  If $C=\emptyset$ and $X$ is genus
  one, then $H^1(\Omega_X)\neq 0$, so this argument fails.  However,
  in this case $\Omega_X$ is trivial,  so it admits the trivial connection.

  Once we have a single connection, the affine structure of the space
  of connections amounts to the fact that the difference of two
  connections with regular singularities is a meromorphic one-form
  with at worst simple poles at the cone points.
\end{proof}

If $\omega$ is a multivalued flat one-form over an affine surface $(X,
\nabla)$, and $\alpha\in H^0(\Omega_X(C))$, then
$e^{\int\alpha}\omega$ is a multivalued flat section for the
connection $\nabla' = \nabla-\alpha$.  This is Veech's description of
this affine structure, which we dubbed the \emph{exponential action}
in \cite{ABW}.

\begin{prop}
  \label{prop:turning difference}
  Consider two affine structures, $\nabla_1$ and $\nabla_2$, on
  $(X,C)$, and let $\theta = \nabla_1-\nabla_2\in \Omega_X(C)$.
  For any closed curve $\gamma\colon[0,1]\to X\setminus C$, we then
  have
  \begin{equation*}
    \tau_{1}(\gamma)-\tau_{2}(\gamma) = - \frac{1}{2\pi i}\int_\gamma\theta,
  \end{equation*}
  where $\tau_i$ represents the turning number with respect to $\nabla_i$.
\end{prop}

\begin{proof}
  Let $D_i$ be the covariant derivative of $\nabla_i$.  We then have
  $D_1(v)-D_2(v) = - \theta(\gamma')v$ for any vector field $v$ along
  $\gamma$.  It follows that
  \begin{align*}
    \tau_{1}(\gamma)-\tau_{2}(\gamma) &= \frac{1}{2\pi i}\int_0^1\left(
                                    \frac{D_1(\gamma'(t))}{\gamma'(t)}-
                                    \frac{D_2(\gamma'(t))}{\gamma'(t)}\right)\, dt \\
                                  &= - \frac{1}{2\pi
                                    i}\int_0^1 \theta(\gamma'(t))\, dt\\
                                  &= - \frac{1}{2\pi
                                    i} \int_\gamma \theta.\qedhere
  \end{align*}
\end{proof}

Recall that a meromorphic one-form $\omega$ on $X$ determines a
translation structure on $X$ and so an affine structure whose cone
points are the zeros and poles of $\omega$.  Denote by $\nabla_\omega$
the connection corresponding to this affine structure, which is
characterized by $\nabla_\omega \omega=0$. Denote by
$\tau_\omega(\gamma)$ the turning number with respect to this affine
structure.

\begin{prop}
  \label{prop:turning de Rham}
  Consider an affine surface $(X, C, \nabla)$, and let $\omega$ be a
  meromorphic one-form on $X$ whose zeros and poles belong to $C$.
  Then for any closed curve $\gamma$ on $X \setminus C$, its turning
  number with respect to $\nabla$ is given by
  \begin{equation*}
    \tau(\gamma) = \tau_\omega(\gamma) - \frac{1}{2\pi i}\int_\gamma\theta,
  \end{equation*}
  where $\theta$ is the meromorphic form $\theta=\nabla\omega/\omega$.
\end{prop}

\begin{proof}
  We have
  \begin{equation*}
    (\nabla-\nabla_\omega)\omega = \nabla \omega = \theta\otimes\omega,
  \end{equation*}
  so $\nabla-\nabla_\omega=\theta$.  The claim then follows from Proposition~\ref{prop:turning difference}.
\end{proof}


%% file: sec_families.tex
\section{Families of affine surfaces}
\label{sec_families}

In this section, we define the notion of a family of affine stable
curves, generalizing smooth affine surfaces, and we construct the
corresponding fine moduli space $\barModA$.   We work
in the category of complex analytic spaces throughout this section.

\paragraph{Relative connections.}

We recall the notion of a relative connection on a family of curves;
see \cite{conrad_riemann_hilbert,deligne_equations_differentielles}
for details.
Given a proper flat family of pointed stable curves
$(\pi\colon\mathcal{X}\to B, \cC)$ (where $\cC = (c_1, \ldots, c_n)$ is a
tuple of disjoint sections $c_i\colon B\to \mathcal{X}$) together with a line
bundle $\mathcal{L}\to\mathcal{X}$, a \emph{relative connection} on
$\mathcal{X}$ with \emph{regular singularities along $\cC$} is a
$\pi^{-1}\mathcal{O}_B$-linear homomorphism of sheaves 
\begin{equation*}
  \nabla\colon \mathcal{L}\to \Omega_{\mathcal{X}/B}(\cC)\otimes_{\mathcal{O}_\mathcal{X}} \mathcal{L},
\end{equation*}
which satisfies for any holomorphic function $f$ and local section
$\sigma$ the Leibniz rule,
\begin{equation*}
  \nabla(f\sigma) = df \otimes \sigma + f \nabla\sigma.
\end{equation*}
Here $\Omega_{\mathcal{X}/B}$ is the relative dualizing sheaf, which
is the relative cotangent bundle $\Omega_{\mathcal{X}} / \pi^*\Omega_B$
when the family is smooth.   As for ordinary connections, the
difference of two relative connections is
$\mathcal{O}_\mathcal{X}$-linear and so a section of 
$\End(\mathcal{L})\otimes \Omega_{\mathcal{X}/B}(\cC)=
\Omega_{\mathcal{X}/B}(\cC)$, or equivalently a section of the extended
Hodge bundle $\pi_*\Omega_{\mathcal{X}/B}(\cC)$.

\begin{definition}
A \emph{family of affine stable curves} is a proper flat family of pointed
  stable curves $(\pi\colon\mathcal{X}\to B,\cC)$, together with a
  connection on $\Omega_{\mathcal{X}/B}$ with regular singularities
  along $\cC$. We will call such a connection an \emph{affine structure} on $(\mathcal{X}, \cC)$.  
\end{definition}

On a stable curve, an affine structure amounts to an affine structure
on each irreducible component with cone points at the nodes, together
with a  condition on the orders of any two cone points glued to a
node, generalizing the ``matching orders'' condition of \cite{BCGGM}.

\begin{prop}
  An affine structure $\nabla$ on a pointed stable curve $(X,C)$ is
  equivalent to an affine structure $\tnabla$ on the normalization
  $p\colon \widetilde{X}\to X$ such that for any two points $n_1, n_2$
  lying over a node $n$ of $X$, the orders of $\tnabla$ satisfy
  \begin{equation}
    \label{eq:matching}
    \ord_{n_1}\tnabla + \ord_{n_2} \tnabla = -2.
  \end{equation}
\end{prop}

\begin{proof}
  Let $C' = p^{-1}(C\cup N)$, where $N$ is the set of nodes of $X$.
  The pullback $p^*\nabla$ is a connection on $p^*\Omega_X = \Omega_{\widetilde{X}}(p^{-1}(N))$ with regular
  singularities along $C'$.  Over any node $n$ of $X$, we have
  \begin{equation*}
    \Res_{n_1}p^*\nabla + \Res_{n_2} p^*\nabla = 0
  \end{equation*}
  by the opposite residue condition for stable forms in $\Omega_X$.
  The restriction of $p^*\nabla$ to $\Omega_{\widetilde{X}}$ defines a
  connection $\tnabla$ whose residues over the nodes is given by
  $\Res_{n_j}\tnabla = \Res_{n_j} p^*\nabla +1$ by \eqref{eq:1} (with
  $g$ having a simple zero at $n_j$).  Equation \eqref{eq:matching}
  then follows from Theorem~\ref{T:DefinitionOfAffine}.
\end{proof}

Given any family of pointed stable curves $(\pi\colon\mathcal{X}\to B, \cC)$,
assigning to an open $U\subset B$ the set of affine
structures over $U$ defines a sheaf of sets $\mathcal{A}_{\mathcal{X}/B}$ over~$B$.
The sheaf $\mathcal{A}_{\mathcal{X}/B}$ is a torsor over the Hodge bundle
$\pi_*\Omega_{\mathcal{X}/B}(\cC)$, meaning that for any affine
structure $\nabla$ over $U$, the map $\omega\mapsto \nabla+\omega$ is
a bijection $\Gamma(U, \Omega_{\mathcal{X}/B}(\cC))\to \Gamma(U, \mathcal{A}_{\mathcal{X}/B})$.

\begin{lemma}
  \label{lem:affine_over_stein}
  If $B$ is a Stein space, then every family of pointed stable curves
  $(\pi\colon\mathcal{X}\to B, \cC)$ has an affine structure.
\end{lemma}

\begin{proof}
  The Atiyah class, defined in \eqref{eq:atiyah_class}, generalizes in
  the obvious way to define the relative Atiyah class $A\in
  H^1(\Omega_{\mathcal{X}/B}(\cC))$, and the desired relative connection
  exists if and only if $A=0$.

  By relative duality,
  \begin{equation*}
    R^1 \pi_* \Omega_{\mathcal{X}/B}(\cC) = \pi_*
    \mathcal{O}_\mathcal{X}(-\cC)^* =0,
  \end{equation*}
  so it follows from the Leray spectral sequence of $\mathcal{X}\to B$
  that
  \begin{equation*}
    H^1(B, \pi_* \Omega_{\mathcal{X}/B}(\cC))\to H^1(\mathcal{X}, \Omega_{\mathcal{X}/B}(\cC))
  \end{equation*}
  is an isomorphism.  If $B$ is Stein, the higher cohomology of any
  coherent sheaf on $B$
  vanishes by Cartan's Theorem~B \cite{Cartan}, and the claim follows.
\end{proof}

\paragraph{The moduli space.}

Recall (for example, see \cite{acgh2}) that a morphism of families of
curves $(\pi'\colon\mathcal{X}'\to B')\to (\pi\colon\mathcal{X}\to B)$
is a pair of morphisms $(f, \bar{f})$ such that the following diagram
commutes, 

\begin{center}    
\begin{tikzcd}
\cX' \arrow[r, "\overline{f}"] \arrow[d, "\pi'"]
& \cX \arrow[d, "\pi"] \\
B' \arrow[r, "f"]
& B
\end{tikzcd}
\end{center}

\noindent and the induced map $\mathcal{X}'\to \mathcal{X}\times_B B'$
is an isomorphism.  We define a morphism of families of affine stable
curves
$(\mathcal{X}'\to B', \nabla')\to (\mathcal{X}\to
B, \nabla)$ to be a morphism $(f, \bar{f})$ of the underlying families
such that $\bar{f}^*\nabla = \nabla'$.  This defines a category
fibered in groupoids $\barModA$ over complex analytic spaces
whose objects are families of affine stable curves.

\begin{theorem}\label{T:UniversalProperty}
  For any family of pointed stable curves $\pi\colon\mathcal{X}\to B$,
  there is a bundle of affine spaces
  $\mathcal{A}_{\mathcal{X}/B}\to B$ modeled on the extended Hodge
  bundle $\pi_* \Omega_{\mathcal{X}/B}(\cC)$ together with an affine
  structure $\nabla_B$ on the pullback family
  $\widetilde{\mathcal{X}} = \mathcal{X}\times_B \mathcal{A}_{\cX/B}$.  This
  family of affine surfaces is characterized by the following
  universal property: given a morphism
  $(f, \bar{f})\colon (\pi'\colon \mathcal{X}'\to B')\to
  (\pi\colon\mathcal{X}\to B)$, an affine structure $\nabla'$ on
  $\mathcal{X}'$ determines a unique lift
  $(\hat{f}, \tilde{f}) \colon (\mathcal{X}'\to B')\to
  (\widetilde{\cX}\to\mathcal{A}_{\mathcal{X}/B})$ such that
  $\tilde{f}^*\nabla_B= \nabla'$.
\end{theorem}

\begin{remark}
  It follows from the universal property that the sheaf of sections of
  $\mathcal{A}_{\mathcal{X}/B}$ coincides with the sheaf of sets
  $\mathcal{A}_{\mathcal{X}/B}$, hence the abuse of notation.
\end{remark}

\begin{proof}
  We first assume $B$ is Stein.  Let $\nabla$ be an affine structure
  on $\mathcal{X}\to B$ given by Lemma~\ref{lem:affine_over_stein}.
  Let
  $\mathcal{A}_{\mathcal{X}/B}= B\times\Gamma(\mathcal{X},
  \Omega_{\mathcal{X}/B}(\cC))$, and let $\omega$ be the tautological
  relative one-form on
  $\widetilde{\mathcal{X}} = \mathcal{X}\times\Gamma(\mathcal{X},
  \Omega_{\mathcal{X}/B}(\cC))$.  We give $\widetilde{\mathcal{X}}$ the
  affine structure
  $\nabla_{\mathcal{X}/B}=\pi_{\mathcal{X}}^*\nabla - \omega$.

  Consider a morphism
  $(f, \bar{f})\colon(\mathcal{X}'\to B')\to(\mathcal{X}\to B)$
  together with an affine structure $\nabla'$ on $\mathcal{X}'$.  The
  form
  \begin{equation*}
    \omega'= \bar{f}^*\nabla-\nabla'
  \end{equation*}
  together
  with the isomorphism
  $\Omega_{\mathcal{X}'/B'}\to\bar{f}^*\Omega_{\mathcal{X}/B}$ determines
  a lift
  \begin{equation*}
    (\hat{f}, \tilde{f}) \colon (\mathcal{X}'\to B')\to
    (\widetilde{\mathcal{X}}\to\mathcal{A}_{\mathcal{X}/B})
  \end{equation*}
  such that
  $\tilde{f}^*\omega = \omega'$.  We check,
  \begin{equation*}
    \tilde{f}^*\nabla_{\mathcal{X}/B}
    =\tilde{f}^*(\pi_\mathcal{X}^*\nabla - \omega)
    =\bar{f}^*\nabla - \omega'
    = \nabla'.
  \end{equation*}

  The general case follows from the general principle that a sheaf covered by representable functors is representable (see \cite[Lemma~26.15.4]{stacks-project} or \cite[Exercise~10.1.H]{foag}).

  To summarize this standard argument, choose a Stein open cover $\{U_i\}$ of $B$,
  with $\mathcal{X}_i= \mathcal{X}|_{U_i}$, and let
  $(\widetilde{\mathcal{X}}_i\to \mathcal{A}_{\mathcal{X}_i/U_i},
  \nabla_{\mathcal{X}_i/U_i})$ be the families constructed above.  We
  construct gluing maps
  $\widetilde{\mathcal{X}}_i|_{U_{ij}}\to \widetilde{\mathcal{X}}_j$
  respecting the affine structures
  by applying the universal property of $\widetilde{\mathcal{X}}_j$ to
  $\widetilde{\mathcal{X}}_i|_{U_{ij}}\to\mathcal{X}_j$.  The cocycle
  condition follows from the uniqueness part of the universal
  property.  We then construct
  $(\mathcal{A}_{\mathcal{X}/B},\nabla_{\mathcal{X}/B})$ by gluing.
  
  Given $(f, \bar{f})\colon(\mathcal{X}'\to B')\to(\mathcal{X}\to B)$ with an affine
  structure on $\mathcal{X}'$, the desired lift follows from applying
  the universal property to the cover $\{f^{-1}(U_i)\}$ of $B'$.  
\end{proof}


\begin{proof}[Proof of Theorem~\ref{T:Families}]
  Apply the previous Theorem to the universal curve over the moduli stack $\barmoduli$.
\end{proof}

\begin{cor}
The preimage of $\overline{\cA}_{g,n}$ over $\cM_{g,n}$ can be identified with the moduli space $\cA_{g,n}$ constructed in \cite{ABW}.
\end{cor}
\begin{proof}
Using the analogue of Theorem \ref{T:DefinitionOfAffine} for families, a relative connection could be described as a line bundle $\cL$ over a family of surfaces $(\cX \ra B)$ so that $\cL \otimes \Omega_{\cX/B}$ is holomorphically trivial. The definition of $\cA_{g,n}$ in \cite{ABW} was the universal family of such line bundles over $\cM_{g,n}$. By Theorem \ref{T:UniversalProperty} the claim holds. 
\end{proof}


%% file: sec_deformations.tex
\section{Deformations of affine surfaces}
\label{sec:deformations}

In this section, we prove Theorem~\ref{T:DerivativeOfHol},
calculating the tangent bundle to $\ModA$, by identifying the space of
first order deformations of an affine surface $(X,C,\nabla)$ with the
first hypercohomology $\mathbb{H}^1(L^\bullet)$ of the two-term
complex $L^\bullet$ of sheaves
$ \xymatrix{ T_X(-C) \ar[r]^{\mathcal{L}_\nabla} & \Omega_X },$ where
the differential is the Lie derivative.

\paragraph{Hypercohomology.}

Suppose that
$\ds{K^\bullet := \hdots \ra \cF^i \xra{d^i} \cF^{i+1} \ra \hdots}$ is
a left-bounded sequence of sheaves on a topological space. Fix an open
cover of the space by acyclic open sets with acyclic intersections and
let $C^i(\cF^j)$ be the space of \v{C}ech $i$-cochains of the sheaf
$\cF^j$.  The \emph{\v{C}ech double complex} $(C^\bullet(K^\bullet),
d, \delta)$ gives rise to its associated \emph{total complex} with degree $n$ term 
\begin{equation}
  \label{eq:total_complex}
  C^n(K^\bullet) := \bigoplus_{i+j = n} C^i(\cF^j)
\end{equation}
and differential
$D^n := \sum_{i+j = n} \delta^i + (-1)^i d^j$. The \emph{hypercohomology} $\bH^\bullet(K^\bullet)$ of
$K^\bullet$ is the cohomology of the total complex.

If $X$ is a Riemann surface and $C\subset X$ a finite set, recall the
\emph{log complex} of $X$ is the two term complex $\ds{\cO_X\xra{d}
\Omega_X(C)}$.  The Algebraic de Rham Theorem
\cite{Grothendieck-AlgebraicDeRham, grhabook} gives an
isomorphism $\mathbb{H}^p(\Omega_X(\log C))\to H^p(X\setminus C;
\cx)$.

\paragraph{Lie derivatives.}

Let $(X, C, \nabla)$ be an affine surface, and $v\in \Gamma(U,T_X(-C))$
a holomorphic vector field vanishing along $C$.  The Lie derivative of
$\nabla$ with respect to $v$ is 
\begin{equation*}
  L_v\nabla \coloneqq \frac{d}{dt} \left(\Phi_v^t\right)^*\nabla\Big|_{t=0} = \lim_{t \rightarrow 0} \frac{\left(\Phi_v^t\right)^*\nabla - \nabla}{t},
\end{equation*}
where $\Phi_v^t$ is the flow defined by the vector field $v$.  Since
the difference of two connections is a holomorphic one-form, the Lie
derivative $L_v\nabla $ is naturally a holomorphic one-form on $X$.
By the Leibniz rule,
\begin{equation}
  \label{eq:leibniz_connection}
  (L_v \nabla) \otimes \omega =L_v (\nabla\omega) - \nabla (L_v \omega).
\end{equation}
Alternatively, the second derivative operator
\begin{equation*}
  D_2(v) = d \nabla^* v
\end{equation*}
naturally outputs a holomorphic one-form (where we implicitly regard $\nabla^* v$
as a holomorphic function via the contraction $T_X \otimes \Omega_X
\to \mathcal{O}_X$).

\begin{prop}
	\label{prop:LieDerivative}
  We have $L_v\nabla = -D_2 v$.
\end{prop}

\begin{proof}
  In local coordinates, let $v = \phi(z) \frac{d}{dz}$ and $\nabla dz =
  \Gamma(z) dz^2$.  In these coordinates, $L_v(\Gamma) = \phi \frac{d
    \Gamma}{d z}$ and by Cartan's formula, 
  \begin{equation*}
    L_v(dz) = \iota_v d(dz) + d(\iota_v dz) = \frac{d\phi}{d z} dz
  \end{equation*}
  Applying the Leibniz rule, it then follows that 
  \begin{align*} 
    (L_v\nabla)dz & = L_v(\nabla dz) - \nabla(L_v dz) \\
    &  = \left( \phi \frac{d \Gamma}{d z} + 2 \Gamma \frac{d \phi}{d z}\right) dz^2 - \left(\frac{d^2\phi}{d z^2} + \Gamma \frac{d \phi}{d z}\right)dz^2 \\
    & = 
    \left(\deriv{\Gamma}{z}\phi + \Gamma \deriv{\phi}{z} -
      \deriv[2]{\phi}{z}\right) dz^2  \\
      & = \left(- d \nabla^* \left( \phi \deriv{}{z} \right) \right)dz.\qedhere
  \end{align*}
\end{proof}


\paragraph{Deformations of affine surfaces.}

Following Voisin \cite{voisinbook}, we first recall the calculation of
the space of first-order deformations of a Riemann surface. Let
$B \coloneqq \Spec\cx[\epsilon]/(\epsilon^2)$. Recall a first-order
deformation of a Riemann surface $X$ is a flat analytic space
$\pi\colon \mathcal{X} \to B$ together with an identification of $X$
with the fiber $\pi^{-1}(\Spec\cx)$ with $X$.  A first-order
deformation of a pointed Riemann surface $(X, C)$ is given by the
additional data of sections $\mathcal{C}$ whose intersection with $X$
is $C$.  A first-order deformation of an affine surface
$(X, C,\nabla)$ is then given by the extra data of a connection
$\widetilde{\nabla}$ on $\Omega_{\mathcal{X}/B}$ with at worst simple
poles along $\mathcal{C}$ whose restriction to $X$ is $\nabla$.  We
write $\Def(X, C,\nabla)$ for the space of first-order deformations of $(X, C,\nabla)$.

Consider the two-term complex of sheaves $L^\bullet$ (in degrees $0$ and
$1$):
\begin{equation*}
  \xymatrix{
    T_X(-C) \ar[r]^-{\mathcal{L}_\nabla} & \Omega_X(C)
  }
\end{equation*}
where the differential is the Lie derivative $\mathcal{L}_\nabla(v) = L_v\nabla$.

By the universal property of $\ModA$, the space of first-order
deformations of $(X, C,\nabla)$ is naturally isomorphic to the tangent
space of $\ModA$ at the corresponding point (see \cite[\S{}VI.1.3]{eisenbud-harris-schemes}).  Theorem~\ref{T:DerivativeOfHol} then immediately follows from the following calculation.

\begin{theorem}
  \label{thm:affine_deformations}
  There is a natural isomorphism $\cD\colon \Def(X, C,\nabla)\to\mathbb{H}^1(L^\bullet)$.
\end{theorem}

Explicitly,  $\mathbb{H}^1(L^\bullet) = \kernel D^1 / \image D^0$,
where
\begin{equation*}
  D^1\colon C^1(T_X(-C))\oplus C^0(\Omega_X(C))\to C^2(T_X(-C))
  \oplus C^1(\Omega_X(C))
\end{equation*}
is defined by
\begin{equation*}
  D^1(\chi, \eta) = (\delta\chi, \mathcal{L}_\nabla(\chi)- \delta\eta),
\end{equation*}
and $D^0\colon C^0(T_X(-C))\to C^1(T_X(-C))\oplus C^0(\Omega_X(C))$ is defined by
\begin{equation*}
  D^0(\chi) = (\delta\chi, \mathcal{L}_\nabla \chi).
\end{equation*}

As preparation, we recall the Kodaira-Spencer isomorphism identifying
$H^1(T_X(-C))$ with the first-order deformations of $(X, C)$.

Let $(\pi\colon \mathcal{X}\to B, \mathcal{C})$ be a first-order
deformation of $(X, C)$.  Choose an open covering $\mathcal{U} =
\{V_\alpha\}$ of $\mathcal{X}$, together with holomorphic
trivializations commuting with the projection to $B$
\begin{equation}
  \label{eq:trivializations}
  F_\alpha\colon V_\alpha \to U_\alpha \times B,
\end{equation}
where $U_\alpha = V_\alpha\cap X$.  These can be chosen to send the
sections $\mathcal{C}$ to horizontal sections $*\times B$.  Let
$F_{\alpha\beta} = F_\alpha\circ F_\beta^{-1}\colon
U_{\alpha\beta}\times B \to U_{\alpha\beta}\times B$, where
$U_{\alpha\beta} = U_\alpha\cap U_\beta$.  We may write
$F_{\alpha\beta}$ as
\begin{equation*}
  F_{\alpha\beta}(z, \epsilon) = (f_{\alpha\beta}(z,\epsilon), \epsilon),
\end{equation*}
where $f_{\alpha\beta}\in\mathcal{O}_{U_{\alpha\beta}}[\epsilon]/(\epsilon^2)$.

Pulling back by
$F_{\alpha\beta}$ defines ring automorphisms
\begin{equation*}
  F_{\alpha\beta}^*\colon \mathcal{O}_{U_{\alpha\beta}}[\epsilon]/(\epsilon^2)\to \mathcal{O}_{U_{\alpha\beta}}[\epsilon]/(\epsilon^2)
\end{equation*}
with $F_{\alpha\beta}^*(\epsilon) = \epsilon$ and for any $g\in
\mathcal{O}_{U_{\alpha\beta}}$,
\begin{equation*}
  F_{\alpha\beta}^* g = (g + \epsilon \chi_{\alpha\beta}(g)).
\end{equation*}
Since $F_{\alpha\beta}^*$ is a ring homomorphism,
$\chi_{\alpha\beta}\colon\mathcal{O}_{U_{\alpha\beta}}\to \cx$ must be
a derivation and so can be regarded as a vector field
$\chi_{\alpha\beta}\in T_X(-C)(U_{\alpha\beta})$.  In other terms,
$\chi_{\alpha\beta}= \pderiv{f_{\alpha\beta}}{\epsilon}\bigr|_{\epsilon=0}\pderiv{}{z}$.
Since
$F_{\alpha\beta}\circ F_{\beta\gamma} \circ F_{\gamma_\alpha} =
\mathrm{Id}$, the $\chi_{\alpha\beta}$ satisfy the cocycle condition
\begin{equation*}
  \chi_{\alpha\beta} + \chi_{\beta\gamma} + \chi_{\gamma\alpha} = 0,
\end{equation*}
and so define a class $\chi$ in $H^1(T_X(-C))$.

\begin{lemma}
  With $F_{\alpha\beta}$ as above, we have for any holomorphic
  one-form $\omega$ on $U_\alpha$ or connection $\nabla$ on
  $\Omega_{U_\alpha}$
  \begin{equation*}
    F_{\alpha\beta}^* \omega = \omega + \epsilon L_{\chi_{\alpha\beta}}\omega
    \quad\text{and}\quad F_{\alpha\beta}^* \nabla = \nabla + \epsilon L_{\chi_{\alpha\beta}}\nabla.
  \end{equation*}
\end{lemma}

\begin{proof}
  We suppress the indices $\alpha,\beta$ in this proof, and work in
  local coordinates $z$ on $U_\alpha$.  We write
  \begin{align*}
    F(z, \epsilon) &= (z + \epsilon g(z), \epsilon)\\
    \omega &= h(z) dz\\
    \nabla dz &= \Gamma(z) dz^2.
  \end{align*}
  In these coordinates, $\chi = g(z) \frac{d}{dz}$.

  We calculate,
  \begin{align*}
    F^* \omega &= h(z + \epsilon g) d(z + \epsilon g)\\
               &= \left(h + \epsilon \deriv{h}{z}g\right)\left(1 + \epsilon \deriv{g}{z}\right) dz\\
               &= \left(h + \epsilon\left(\deriv{g}{z}h + g \deriv{h}{z}\right)\right)dz\\
               &= \omega + \epsilon L_\chi \omega,
  \end{align*}
  which proves the first claim.

  On the other hand let $(F^*\nabla) dz = \Delta(z) dz^2$.  Then, $$(F^*\nabla)F^*dz = F^*(\nabla dz).$$ Expanding the left hand side gives us that $$(F^*\nabla)F^*dz = (F^*\nabla)\left(dz + \epsilon \deriv{g}{z}dz\right) =  \left(\Delta  + \epsilon \deriv[2]{g}{z}  + \epsilon \deriv{g}{z} \Delta \right) dz^2, $$ and expanding the right hand side gives us that 

  \begin{align*}
    F^*(\nabla dz) = F^*(\Gamma dz^2) &= \left(\Gamma + \epsilon
    \deriv{\Gamma}{z}g\right)\left(1 + \epsilon \deriv{g}{z}\right)^2
 dz^2 \\
    &= \left( \Gamma + \epsilon \left(\deriv{\Gamma}{z} g+ 2 \Gamma \deriv{g}{z} \right)\right) dz^2.
  \end{align*}
  
  Equating the two expressions when $\epsilon = 0$, we see that $\Delta(z) = \Gamma(z) + \epsilon R(z)$ for some function $R$. Using this fact, equating the two expressions, and solving for $\Delta dz^2$ gives us that $$\Delta dz^2= \left(\Gamma + \epsilon\left(\deriv{\Gamma}{z} g + \Gamma \deriv{g}{z} - \deriv[2]{g}{z}\right)\right) dz^2 = \nabla dz + \epsilon (L_\chi \nabla) dz,$$
  
  with the expanded form of $L_\chi \nabla$ following from the proof of Proposition \ref{prop:LieDerivative}. Since $(F^* \nabla)dz = \Delta dz^2$, the second claim follows. 
\end{proof}

Consider a first-order deformation $(\cX, \cC, \tnabla)$ of an affine
surface $(X, C, \nabla)$.  Choose an open cover $\{U_\alpha\}$ of $X$
together with trivializations $F_\alpha$ as in
\eqref{eq:trivializations}, which we have defines the class $\chi = (\chi_{\alpha \beta})$ in $C^1(T_X(-C))$ associated to the deformation $(\cX, \cC)$ of $(X, C)$.  Choose moreover on each $U_\alpha$ a nowhere
vanishing holomorphic one-form $\omega_\alpha$, and suppose
$\nabla\omega_\alpha = \eta_\alpha\otimes\omega_\alpha$.  The pullback
$\tnabla_\alpha = (F_\alpha)_* \tnabla$ is a connection over $U_\alpha\times B$
characterized by a meromorphic form $\eta_\alpha'$ on $U_\alpha$, with
at worst simple poles along $C$, so that
\begin{equation*}
  \tnabla_\alpha \omega_\alpha = (\eta_\alpha + \epsilon
  \eta_\alpha')\otimes \omega_\alpha.
\end{equation*}
We define the isomorphism $\cD$ of Theorem~\ref{thm:affine_deformations} by $\cD (\cX, \cC, \tnabla) = (\chi, \eta')$.

\begin{lem}\label{L:ExistenceOfCocycle}
The class $(\chi, \eta') \in C^1(T_X(-C)) \oplus C^0(\Omega_X(C))$ is in the kernel of $D^1$.  Moreover, any such class arises from some first-order deformation of $(X, C, \nabla)$.
\end{lem}
\begin{proof}
  For the connections $\tnabla_\alpha$ to glue to a global connection on
  $\mathcal{X}$ amounts to
   \begin{equation}
     \label{eq:gluing_condition} F_{\alpha\beta}^* \tnabla_\alpha
\omega_\alpha = \tnabla_\beta (F_{\alpha\beta}^* \omega_\alpha),
   \end{equation}
   which we claim is equivalent to
   $\mathcal{L}_\nabla(\chi) - \delta \eta = 0$.  Since
   $\delta\chi=0$ exactly when $\chi$ arises from a first-order
   deformation of $(X, C)$, it follows from this claim that
   $D^1(\chi, \eta')=(\delta\chi, \mathcal{L}_\Delta(\chi) - \delta
   \eta)=0$ exactly when $(\chi, \eta')$ arises from a first-order deformation of $(X, C, \nabla)$, as desired.

   The left-hand-side of \eqref{eq:gluing_condition} simplifies to
  \begin{align*}
    &F_{\alpha\beta}^*(\eta_\alpha\otimes \omega_\alpha + \epsilon
    \eta_\alpha'\otimes \omega_a)\\
    &= (\eta_\alpha + \epsilon L_{\chi_{\alpha\beta}} \eta_\alpha) \otimes (\omega_\alpha + \epsilon L_{\chi_{\alpha\beta}} \omega_\alpha) + \epsilon \eta_\alpha' \otimes \omega_\alpha\\
    &= \eta_\alpha \otimes \omega_\alpha + \epsilon(\eta_\alpha\otimes L_{\chi_{\alpha\beta}}\omega_\alpha + L_{\chi_{\alpha\beta}}\eta_\alpha\otimes\omega_\alpha + \eta_\alpha' \otimes \omega_\alpha).
  \end{align*}
  Letting $g_{\alpha \beta}= \omega_\alpha/\omega_\beta$, the right-hand-side of \eqref{eq:gluing_condition} becomes
  \begin{equation*}
    \tnabla_\beta (\omega_\alpha + \epsilon L_{\chi_{\alpha\beta}} \omega_\alpha) 
    = (d g_{\alpha\beta}+ g_{\alpha\beta}\eta_\beta) \otimes \omega_\beta + \epsilon (\eta_\beta'\otimes \omega_\alpha + \nabla L_{\chi_{\alpha\beta}} \omega_\alpha). 
  \end{equation*}
  Subtracting these two equations gives
  \begin{equation*}
    \epsilon(\eta_\alpha\otimes L_{\chi_{\alpha\beta}}\omega_\alpha + L_{\chi_{\alpha\beta}}\eta_\alpha\otimes\omega_\alpha - \nabla L_{\chi_{\alpha\beta}} \omega_\alpha + (\eta_\alpha'-\eta_\beta')\otimes \omega_\alpha ) = 0, 
  \end{equation*}
  where we have used that the constant term vanishes because $\nabla$ is a global connection on $X$.  Using Equation \eqref{eq:leibniz_connection}, this becomes 
  \begin{equation*}
     \mathcal{L}_\nabla(\chi_{\alpha\beta}) + \eta_\alpha' - \eta_\beta'=0.
    \qedhere
  \end{equation*}
\end{proof}

\begin{lem}\label{L:UniquenessOfCocycle}
  A different first-order deformation $(\overline{\cX}, \overline{\cC}, \tbnabla)$ of $(X, C, \nabla)$ is isomorphic to $(\cX, \cC, \tnabla)$ if and only if $(\overline{\chi}, \overline{\eta}')=\cD(\overline{\cX}, \overline{\cC}, \tbnabla)$ differs from $(\chi,\eta')$ by a coboundary $D^0(\psi)$.
\end{lem}
\begin{proof}
  We denote analogous objects for the deformation
  $(\overline{\cX}, \overline{\cC}, \tbnabla)$ by an overline, such as
  transition functions $\overline{F}_{\alpha\beta}$ and cocycles
  $\overline{\chi}_{\alpha\beta}$.

  Constructing an isomorphism of first-order deformations amounts to constructing $G_\alpha\colon U_\alpha\times B\to U_\alpha\times B$ which commute with the projection to $B$, respecting marked sections, and such that
  \begin{align}
    \label{eq:FG}
    F_{\alpha\beta} G_\beta &=G_\alpha \overline{F}_{\alpha\beta} \quad\text{and}\\
    \label{eq:pullbackbyG}
    G_\alpha^*\tnabla_\alpha &= \tbnabla_\alpha.
  \end{align}
  The $G_\alpha$ are defined by holomorphic vector fields
  $\psi_\alpha\in T_X(-C)(U_\alpha)$ so that
  \begin{equation*}
    G_\alpha^* h = (h + \epsilon \psi_\alpha(g)),
  \end{equation*}
  for every $h$ holomorphic on $U_\alpha$.  In these terms, \eqref{eq:FG} amounts to
  \begin{equation}
    \label{eq:coboundary1}
    \overline{\chi} - \chi = \delta (\psi).
  \end{equation}
  
  Equation \eqref{eq:pullbackbyG} is equivalent to
  \begin{equation}
    \label{eq:pullbackbyG2}
    G_\alpha^* \tnabla_\alpha \omega_\alpha = \tbnabla_\alpha
    G_\alpha^* \omega_\alpha.
  \end{equation}
  The left hand side of \eqref{eq:pullbackbyG2} simplifies to
  \begin{equation*}
    G_\alpha^*\left( (\eta_\alpha + \epsilon \eta_\alpha')\otimes
    \omega_\alpha \right) = (\eta_\alpha + \epsilon(L_{\psi_\alpha}
    \eta_\alpha + \eta_\alpha'))\otimes (\omega_\alpha + \epsilon L_{\psi_\alpha}\omega_\alpha),
  \end{equation*}
  and the right hand side simplifies to
  \begin{equation*}
    \tbnabla_\alpha(\omega_\alpha + \epsilon L_{\psi_\alpha}
    \omega_\alpha) = (\eta_\alpha + \epsilon
    \overline{\eta}_\alpha')\otimes \omega_\alpha + \epsilon
    \tnabla_\alpha L_{\psi_\alpha}\omega_\alpha.
  \end{equation*}
  Subtracting these gives
  \begin{align*}
    \epsilon(\overline{\eta}_\alpha' - \eta_\alpha')\otimes \omega_\alpha &= \epsilon(L_{\psi_\alpha} (\eta_\alpha \otimes \omega_\alpha)  - \nabla L_{\psi_\alpha}
    \omega_\alpha ) \\
                                                                          &=\epsilon( L_{\psi_\alpha} (\nabla \omega_\alpha) - \nabla (L_{\psi_\alpha} \omega_\alpha) )  \\
                                                                          &= \epsilon (L_{\psi_\alpha}\nabla)\otimes \omega_\alpha,
  \end{align*}
  where the last equality follows from \eqref{eq:leibniz_connection}.  
  It follows that \eqref{eq:pullbackbyG} is equivalent to
  \begin{equation}
    \label{eq:coboundary2}
    \mathcal{L}_\nabla(\psi_\alpha)=  \overline{\eta}_\alpha'-\eta_\alpha'.
  \end{equation}
  
  Combining equations \eqref{eq:coboundary1} and \eqref{eq:coboundary2}, we see that defining an isomorphism of deformations amounts to solving the coboundary equation,  $(\overline{\chi}, \overline{\eta}')-(\chi,\eta') = (\delta
  (\psi), \mathcal{L_{\nabla}(\psi)}) = D^0(\psi)$. 
\end{proof}

\begin{proof}[Proof of Theorem~\ref{thm:affine_deformations}]
  Lemmas~\ref{L:ExistenceOfCocycle} and~\ref{L:UniquenessOfCocycle} together show that $\cD$ is a well-defined bijection.
\end{proof}

%% file: sec_holonomy_new.tex
\section{The derivative of holonomy}
\label{sec:derivativeofholonomy}

In this section, we prove Theorem~\ref{thm:holonomy_derivative_intro},
calculating the derivative of $\THol$.  The idea of the proof is to
use Proposition~\ref{prop:turning de Rham} to represent the turning
number of any curve $\gamma$ as a constant minus $\frac{1}{2\pi i}\int_\gamma
\theta$, where $\theta = \nabla\zeta/\zeta$ for $\zeta$ a meromorphic one-form.
Differentiating turning numbers then amounts to differentiating the de
Rham cohomology of $\theta$, or in other words, applying
the Gauss-Manin connection to $\theta$.  We start by
recalling a formula for the Gauss-Manin connection (see \cite{katzoda}
or \cite[Theorem~7.7]{deligne_equations_differentielles} for more
general calculations).

Let $\Omega'\Teich$ denote the \Teichmuller space of pointed surfaces
$(X,C)$  with a meromorphic one-form $\omega$ with at worst
 simple poles along $C$, and consider the map $\dR\colon \Omega'\Teich \to H^1( \Sigma
\setminus P, \cx)$ sending a form to its de Rham cohomology class.

Given $(X, C, \omega)\in\Omega'\Teich$,  consider the two-term complex of sheaves $M^\bullet$ (in degrees $0$ and
$1$):
\begin{equation*}
  \xymatrix{
    T_X(-C) \ar[r]^{\mathcal{L}_\omega}  & \Omega_X(C)
  },
\end{equation*}
whose differential is the Lie derivative $\mathcal{L}_\omega(v) =
L_v\omega$.

\begin{prop}
  \label{prop:GaussManin}
  Given a form $(X, C, \omega)\in \Omega'\Teich$, there is a natural
  isomorphism between the tangent space $T_{(X, C,
    \omega)}\Omega'\Teich$ and the first hypercohomolgy
  $\mathbb{H}^1(M^\bullet)$.  In terms of this isomorphism, the
  derivative of $\dR$ at $(X, C, \omega)$ is the map
  $\mathbb{H}^1(M^\bullet)\to \mathbb{H}^1(\Omega_X(\log C))$ induced
  by the following morphism of complexes,
  \begin{equation*}
    \xymatrix{
      \Omega_X(C) \ar[r]^{\mathrm{id}} & \Omega_X(C) \\
      T_X(-C) \ar[r]^{i_\omega}\ar[u]^{\mathcal{L}_\omega}  & \mathcal{O}_X \ar[u]^d
    }, 
  \end{equation*}
  where the bottom arrow is the interior product $i_\omega(v) = \iota_v(\omega)$.
\end{prop}

\begin{proof}
  Consider a first-order deformation $(\pi\colon\mathcal{X}\to B,
  \mathcal{C}, \tilde{\omega})$ of $(X, C, \omega)$.  Following the notation for
  deformations in \S\ref{sec:deformations}, the deformation
  $(\mathcal{X}, \mathcal{C})$ is determined by a cocycle
  $\chi\in H^1(T_X(-C))$, and $\widetilde{\omega}$ is represented in local
  coordinates as
  \begin{equation*}
    (F_\alpha)_*\widetilde{\omega} = \omega + \epsilon \omega_\alpha',
  \end{equation*}
  where $\omega_\alpha'\in \Gamma(U_\alpha, \Omega_X(C))$.  This
  construction assigns a cocycle
  $(\chi, \omega')\in C^1(T_X(-C))\oplus C^0(\Omega_X(C))$ to every
  first-order deformation of $(X, C, \omega)$.  The condition for
  these forms to glue is that
  \begin{equation*}
    \omega + \epsilon \omega_\beta' = F_{\alpha\beta}^*(\omega +
    \epsilon \omega_\alpha') = \omega +
    \epsilon(\mathcal{L}_\omega(\chi_{\alpha\beta}) + \omega_\alpha'),
  \end{equation*}
  which is equivalent to $D^1(\chi, \omega')= (\delta\chi,
  \mathcal{L}_\omega(\chi) - \delta \omega')=0$.   Given a different
  choice of trivializations $\overline{F}_\alpha$ of $\mathcal{X}$,
  and with $(\overline{F}_\alpha)_*\widetilde{\omega} = \omega +
  \epsilon \overline{\omega}_\alpha'$,  let $G_\alpha =
  F_\alpha\overline{F}_\alpha^{-1}$.  We have
  \begin{equation*}
    \omega + \epsilon \overline{\omega}_\alpha' = G_\alpha^*(\omega +
    \epsilon \omega_\alpha') = \omega +
    \epsilon(\mathcal{L}_\omega(\psi_\alpha)+ \omega_\alpha' ).
  \end{equation*}
  This is equivalent to
  $(\overline{\chi}, \overline{\omega}')-(\chi,\omega') = (\delta
  \psi, \mathcal{L}_{\omega}(\psi)) = D^0(\psi)$, which gives that
  our correspondence between first-order deformations and
  $\mathbb{H}^1(M^\bullet)$ is well-defined and bijective.

  It remains to show that the cocycle $(i_\omega(\chi), \omega')$
  represents the derivative of $\dR$ applied to the tangent vector
  represented by our above first-order deformation.  Since $\Omega'\Teich$
  is smooth, every tangent vector is tangent to a holomorphic disk, so
  we may pass from our first-order deformation to a deformation over a
  holomorphic disk.  We will continue to use the same notation from
  the previous paragraph and \S\ref{sec:deformations} with the
  understanding that the base $B$ is now a disk.

  By \cite[Proposition~6.2.3]{HubbardBook}, the family $\mathcal{X}\to B$ (after possibly decreasing the
  radius of $B$) has a horizontally analytic trivialization
  $\overline{F}\colon \mathcal{X}\to\ X \times B$, meaning that $\overline{F}$ is
  smooth, and the restriction of $\overline{F}^{-1}$ to any
  horizontal slice $\{z\}\times B$ is holomorphic.  Write
  $\overline{F}_\alpha(z,\epsilon)$
  for the restriction of $\overline{F}$ to $V_\alpha$.  As
  before, define $G_\alpha(z,\epsilon)=(g_\alpha(z,\epsilon),
  \epsilon)$ by $G_\alpha = F_\alpha\overline{F}_\alpha^{-1}$, and
  let $\psi_\alpha = \pderiv{g_\alpha}{\epsilon}|_X$, a smooth vector
  field on $U_\alpha$.  As in \eqref{eq:coboundary1}, we have
    $-\chi = \delta(\psi)$.      
 Transporting $\omega$ to $X\times B$ by the smooth trivialization
 gives
 \begin{equation*}
   G_\alpha^*(\omega+\epsilon\omega_\alpha') = \omega +
   \epsilon(L_{\psi_\alpha} + \omega_\alpha') + O(\epsilon^2).
 \end{equation*}
 Since $\delta(L_\psi(\omega) + \omega') = - L_\chi(\omega) +
 \delta\omega'=0$, the forms $\eta_\alpha=L_{\psi_\alpha}(\omega) + \omega_\alpha'$ on
 $U_\alpha$ glue to a global smooth $1$-form $\eta$ on $X$ which represents
 the derivative of $\dR$ along $B$.

 Now, consider the complex of sheaves  $\cA^k_X(C)$ on $X$, where
 $\cA^k_X(C)$ assigns to an open set $U$ the smooth $k$-forms on
 $U\setminus C$.  The morphism of complexes $\Omega^\bullet_X\to
 \cA^\bullet_X$ is a quasi-isomorphism and so the induced map on
 hypercohomology is an isomorphism (see \cite[p.~450]{grhabook})).  To show that the class of
 $(i_\omega(\chi), \omega')$ in $\mathbb{H}^1(\Omega^\bullet_X)$
 represents the class of $\eta$ in de Rham cohomology, it then suffices to
 show that $(i_\omega(\chi), \omega')\sim(0, L_\psi(\omega)+\omega')$ in $\mathbb{H}^1(\cA^\bullet_X)$, as is
 shown by the following computation:
 \begin{multline*}
   (0, L_\psi(\omega) + \omega') -(i_\omega(\chi), \omega') =
   (-i_\omega(\chi), L_\psi(\omega))= \\
   (\delta i_\omega(\psi),  d i_\omega(\psi)) = D^0(i_\omega(\psi))
   \qedhere
 \end{multline*}
\end{proof}

\begin{proof}[Proof of Theorem~\ref{thm:holonomy_derivative_intro}]
  Let $(\cX, \cC, \tnabla)$ be a first order deformation of $(X, C,
  \nabla)$ representing a tangent vector $V$ to $\TA$ at $X$.  Consider a relative meromorphic form $\zeta$ on $\cX$
  with zeros and poles contained in $\cC$, and let $\theta =
  \nabla\zeta/\zeta$.

  Given a closed curve $\gamma$ on $X\setminus C$.  By
  Proposition~\ref{prop:turning de Rham} (whose proof applies to
  families without change),
  \begin{equation*}
    \tau(\gamma)=  \tau_\zeta(\gamma) - \frac{1}{2\pi i}\int_\gamma\theta.
  \end{equation*}
  Since $\tau_\zeta(\gamma)$ is integral, it is constant on $B$.  It
  follows that the derivative $D_V\THol$ may be computed as
  $-\frac{1}{2\pi i}$  times the derivative of $\dR$ along the vector
  associated to the first-order deformation $(\cX, \cC, \theta)$.

  Following the notation from \S\ref{sec:deformations}, we represent
  $\zeta$ in local coordinates as
  \begin{equation*}
    (F_\alpha)_* \zeta = (h_\alpha + \epsilon h_\alpha')\omega_\alpha,
  \end{equation*}
  where $h_\alpha$ and $h_\alpha'$ are meromorphic functions on
  $U_\alpha$.  Since these forms glue to $\zeta$, we have 
  \begin{equation*}
    F_{\alpha\beta}^*(h_\alpha + \epsilon h_\alpha')\omega_\alpha=(h_\beta + \epsilon h_\beta')\omega_\beta,
  \end{equation*}
  which is equivalent to the system,
  \begin{align*}
    h_\alpha\omega_\alpha &= h_\beta\omega_\beta\\
    L_{\chi_{\alpha\beta}}(h_\alpha\omega_\alpha) &=
                                                    h_\beta'\omega_\beta -h_\alpha'\omega_\alpha.
  \end{align*}
  Combining these gives
  \begin{equation}
    \label{eq:fuckthisequation}
    \frac{h_\beta'}{h_\beta} - \frac{h_\alpha'}{h_\alpha} = \frac{L_{\chi_{\alpha\beta}}(h_\alpha\omega_\alpha)}{h_\alpha\omega_\alpha} = \frac{L_{\chi_{\alpha\beta}}h_\alpha}{h_\alpha} + \frac{L_{\chi_{\alpha\beta}} \omega_\alpha}{\omega_\alpha}.
  \end{equation}

  Similarly, we represent $\theta$ in local coordinates as
  \begin{equation*}
    (F_\alpha)_*\theta = \frac{\nabla((h_\alpha + \epsilon
      h_\alpha')\omega_\alpha)}{(h_\alpha + \epsilon
      h_\alpha')\omega_\alpha} = \frac{d h_\alpha}{h_\alpha}+
    \eta_\alpha +
    \epsilon\left(d\left(\frac{h_\alpha'}{h_\alpha}\right) + \eta_\alpha'\right).
  \end{equation*}
  By Proposition~\ref{prop:GaussManin}, the derivative of $\dR$ is
  represented in $\mathbb{H}^1(\Omega_X(\log C))$ by
  \begin{equation*}
    \left(\iota_{\chi_{\alpha\beta}}\left(\frac{dh_\alpha}{h_\alpha} +
        \eta_\alpha \right), d\left(\frac{h_\alpha'}{h_\alpha}\right)+ \eta_\alpha'\right),
  \end{equation*}
  which we claim is cohomologous to $(-\nabla^*\chi_{\alpha\beta},\eta_\alpha')$.  We calculate
  \begin{align*}
    &\left(\iota_{\chi_{\alpha\beta}}\left(\frac{dh_\alpha}{h_\alpha} + \eta_\alpha \right), d\left(\frac{h_\alpha'}{h_\alpha}\right)+
      \eta_\alpha'\right) - (-\nabla^*\chi_{\alpha\beta},\eta_\alpha') - D^0\left(\frac{h_\alpha'}{h_\alpha}\right)\\
    &= \left(\frac{L_{\chi_{\alpha\beta}}h_\alpha}{h_\alpha} + \iota_{\chi_{\alpha\beta}}\eta_\alpha  + \nabla^*\chi_{\alpha\beta} - \frac{h_\beta'}{h_\beta} + \frac{h_\alpha'}{h_\alpha},0\right)\\
    &= \left(-\frac{L_{\chi_{\alpha\beta}}\omega_\alpha}{\omega_\alpha} + \iota_{\chi_{\alpha\beta}}\eta_\alpha  + \nabla^*\chi_{\alpha\beta},0\right),
  \end{align*}
  where the last equality uses \eqref{eq:fuckthisequation}.  That the last line is $0$ follows from
  \begin{align}\label{eq:liederivativeofomega}
    L_{\chi_{\alpha\beta}}\omega_\alpha &= d\langle\chi_{\alpha\beta},
                    \omega_\alpha\rangle\\ \nonumber
    &=\langle\nabla^*\chi_{\alpha\beta}, \omega_\alpha\rangle +
      \langle\chi_{\alpha\beta}, \nabla\omega_\alpha\rangle\\ 
    &= \langle\nabla^*\chi_{\alpha\beta}, \omega_\alpha\rangle+
      \langle\chi_{\alpha\beta},
      \eta_\alpha\otimes\omega_\alpha\rangle.\qedhere \nonumber
  \end{align}

\end{proof}

%% file: sec_stuff2.tex
\section{The coderivative formula}
\label{sec:submersion}

In this section, we prove Theorem~\ref{T:coderivative}, computing the
coderivative of $\THol$, which completes the proof of
Theorem~\ref{T:Veechsubmersion}.  We begin by calculating the
hypercohomology $\bH^1(L^\bullet)$.

\begin{lem}\label{L:HypercohomologyOfLComputation}
We have $\bH^0(L^\bullet) = \bH^2(L^\bullet) = 0$, and there is a short exact sequence
\[ 0 \ra H^0(\Omega_X(C)) \ra \bH^1(L^\bullet) \ra H^1(T_X(-C)) \ra  0 \]
\end{lem}
\begin{proof}
Consider the following short exact sequence of complexes, 
\begin{center}
\begin{tikzcd}
    0 \arrow{r}  &  \Omega_X(C)  \arrow{r} & \Omega_X(C)  \arrow{r}  & 0  \arrow{r}  & 0  \\
     0 \arrow{r} \arrow{u} & 0  \arrow{r}  \arrow{u} & T_X(-C) \arrow{r}  \arrow{u} & T_X(-C) \arrow{r}  \arrow{u}  & 0. \arrow{u}  \\
\end{tikzcd}
\end{center}
It induces a long exact sequence on hypercohomology, which begins
\[ 0 \ra \bH^0(L^\bullet) \ra H^0(T_X(-C)) = 0 \]
which shows that $\bH^0(L^\bullet) = 0$. The sequence continues 
\[ 0 \ra H^0(\Omega_X(C)) \ra \bH^1(L^\bullet) \ra H^1(T_X(-C)) \ra H^1(\Omega_X(C)) = 0 \]
where the final equality is via Serre duality.
The long exact sequence ends with 
\[ 0 \ra \bH^2(L^\bullet) \ra H^2(T_X(-C)) =0  \]
showing that $\bH^2(L^\bullet) = 0$.
\end{proof}

We obtain the following Corollary, which is also a consequence of the fact that
$\ModA\to\moduli$ is an affine bundle.

\begin{cor}
The map from $\cA_{g,n}$ to $\cM_{g,n}$ that forgets the affine structure is a submersion.
\end{cor}
\begin{proof}
Using the notation the previous section, a first order deformation of $(X, \nabla)$ determines an element $(\chi_{\alpha \beta}, \eta_\alpha) \in C^1(T_X(-C)) \oplus C^0(\Omega_X(C))$. Projection to the first factor induces the map $\bH^1(L^\bullet) \ra H^1(T_X(-C))$ given in the sequence in Lemma \ref{L:HypercohomologyOfLComputation}. This is precisely the derivative of the map $\cA_{g,n} \ra \cM_{g,n}$ and it is a surjection by Lemma \ref{L:HypercohomologyOfLComputation}.
\end{proof}

\begin{proof}[Proof of Theorem~\ref{T:coderivative}]
  Consider the  short exact sequence of complexes,
  \begin{center}
    \begin{tikzcd}
      0 \arrow{r}  &  \Omega_X(C)  \arrow{r} & \Omega_X(C)  \arrow{r}  & 0  \arrow{r}  & 0  \\
      0 \arrow{r} \arrow{u} & 0  \arrow{r}  \arrow{u} & \cO_X \arrow{r}  \arrow{u} & \cO_X \arrow{r}  \arrow{u}  & 0., \arrow{u}  \\
    \end{tikzcd}
  \end{center}
  Mapping it to the one in the proof of Lemma
  \ref{L:HypercohomologyOfLComputation} induces the following
  morphism of short exact sequences
  \begin{center}
    \begin{tikzcd}
      0 \arrow{r} & H^0(\Omega_X(C)) \arrow{r} \arrow{d}{-\frac{1}{2\pi i}\mathrm{id}} & \bH^1(L^\bullet) \arrow{r} \arrow{d}{D \THol} &   H^1(T_X(-C)) \arrow{r} \arrow{d}{\frac{1}{2\pi i}\nabla^*} & 0  \\
      0 \arrow{r} & H^0(\Omega_X(C)) \arrow{r} & \bH^1(\Omega_X^\bullet(C)) \arrow{r}  &   H^1(\cO_X) \arrow{r} & 0. 
    \end{tikzcd}
  \end{center}

  Using the notation of the previous two sections, let $U_\alpha$ be
  an open cover of $X$ and let $(v_{\alpha \beta})$ be any closed
  $1$-cocycle in $C^1(T_X(-C))$. Let $\omega$ be a holomorphic
  $1$-form and notice that
  $\langle \omega, v_{\alpha \beta} \rangle := (\omega(v_{\alpha
    \beta}))$ is $1$-cocycle in $C^1(\cO_X)$. The long exact sequence
  associated to the short exact sequence
  $0 \ra \bC_X \ra \cO_X \xrightarrow{d} \Omega_X \ra 0$ contains the segment
\[ H^1(\cO_X) \xrightarrow{d} H^1(\Omega_X) \ra H^2(\bC_X) \ra 0 \]
where the middle arrow is an isomorphism and hence the first arrow is the zero map. Therefore, $0 = d\langle \omega, v_{\alpha \beta} \rangle = \langle \nabla \omega, v_{\alpha \beta} \rangle + \langle \omega, \nabla^* v_{\alpha \beta} \rangle.$  It follows that the dual of the right arrow in the above diagram is $-\frac{1}{2\pi i}\nabla$.
\end{proof}

%% file: sec_stuff.tex
\section{The isoresidual foliation}
\label{sec:isoresidual}

In this section, we prove Theorem~\ref{T:SubmersionWithResidue} and
give an alternate proof of Theorem~\ref{T:Veechsubmersion} by
analyzing a spectral sequence containing $D\THol$ and $D\Res_\Gamma$ as
differentials.

\paragraph{Translation vector fields.}

As preparation,
we analyze the kernel and cokernel of $\nabla^*\colon T_X(-C)\to \mathcal{O}_X$.
We define the sheaf of \emph{translation vector fields} on $X$ to be
the kernel $\trans= \ker \nabla^*$, and we calculate the cokernel in
the proposition below.

Let $W$ be the one-dimensional vector space of germs of flat $1-$forms
along the tree $\Gamma$, and let $W^*_{P_\bZ}$ denote the sheaf $\bigoplus_{p\in P_\bZ}W^*_p$, where
$W^*_{p}$ is the skyscraper sheaf at $p$ with value $W^*$.   
Let $R_\Gamma\colon \mathcal{O}_X\to W_{P_\bZ}$ denote the
homomorphism sending a function $f$ defined
near an integral pole $p$ to the linear map  $\omega\mapsto\Res_{p} f \omega$. 

\begin{prop}\label{P:TransExact}
  The kernel $\trans$ of $\nabla^*$ is
  a sheaf of $\cx$-vector spaces whose stalks are $1$-dimensional over
  the ordinary points and the set $P_\zed$ of poles of integral order
  and are $0$-dimensional elsewhere.  

  The map $R_\Gamma$ induces an isomorphism  $\mathrm{coker}(\nabla^*)\to W^*_{P_\bZ}$.  In other
  words, the following sequence is exact:
  \begin{equation}
    \label{eq:thissequenceisexact}
    \xymatrix{
      0\ar[r] &\trans\ar[r] & T_X(-C) \ar[r]^{\nabla^*}&\mathcal{O}_X \ar[r]^{R_\Gamma} & W^*_{P_\bZ} \ar[r]& 0}
  \end{equation}
\end{prop}

\begin{proof}
    Away from cone points, this follows immediately from the existence
  and uniqueness of solutions to first order ODEs.

  Consider a cone point $p$ of order $k$.  In local coordinates
  centered at $p$, we have
  $\nabla^*\deriv{}{z}= \Gamma= \frac{k}{z} + \Gamma_0$, where $\Gamma_0$
  is holomorphic near $0$.  In these coordinates, a flat one-form
  $\omega$ is given by
  $e^{\int \Gamma\, dz}\, dz= z^ke^{\int \Gamma_0\, dz}\,
  dz$, and $\omega^{-1}$ is a flat vector field near $p$, which represents a
  single-valued vector field vanishing at $p$ exactly when $k$ is a
  negative integer.

  For exactness at $\mathcal{O}_X$, consider the ODE $\nabla^* (v) = g$, where $g$ is 
  holomorphic near $0$.  A solution is given
  explicitly by
  \begin{equation*}
    v = \omega^{-1} \int  g \omega
  \end{equation*}
  When $k$ is not
  integral, while the integral exists only as a multivalued function,
  the resulting $f$ is single-valued and vanishes at $0$ as desired.
  When $k$ is integral, the integral exists only when $\Res_{p} g\omega
  = 0$.  In other words, the image of $\nabla^*$ at $p$ is the kernel of
  $g\mapsto \Res_{p} g\omega$.
\end{proof}

There is a functor from the category of sheaves on a topological space
$U$ to the category of abelian groups that sends a sheaf $\cF$ to
$H^0_c(U, \cF)$, the compactly supported sections of $\cF$ on
$U$. \emph{Cohomology with compact support} are the associated right
derived functors. If $U$ is an open and dense subset of a compact
space $X$, with inclusion map $\iota: U \hookrightarrow X$, then
$H^i_c(U, \cF) \cong H^i(X, \iota_! \cF)$ where $\iota_! \cF$ is the
\emph{extension by zero sheaf}, which is the sheaf whose sections on
an open set $V$ are the elements of $H^0(V \cap U, \cF)$ whose
supports don't intersect $X \setminus U$. For a proof see
\cite[Proposition 4.7.1]{Harder-AG}. Another description of $\iota_!
\cF$ is the unique sheaf on $X$ whose stalks at $x \in X$ are the same
as the stalk of $\cF$ when $x \in U$ and are $0$ otherwise (see \cite[Lemma~59.70.4]{stacks-project}). 

\begin{lem}\label{L:H2Trans}
$h^2(\trans) = 1$ if $X$ is a translation surface and $0$ otherwise. 
\end{lem}
\begin{proof}
We have a short exact sequence,
\[ 0 \ra \iota_! \trans[X \setminus C] \ra \trans \ra
  \cS \ra 0, \] where $\iota \colon X \setminus C \ra X$ is the
inclusion map, and $\cS$ is a sum of skyscraper sheaves
supported over $C$. The long exact sequence on cohomology
yields
\[ H^2(\trans) \cong H^2(\iota_! \trans[X \setminus C])
  \cong H_c^2(\trans[X \setminus C]) \cong H^0(\trans[X
    \setminus C]^\vee), \]
where the last isomorphism is by Poincar\'e
duality for local coefficient systems (see \cite[Theorem
4.8.9]{Harder-AG}). Note that
$h^0(\trans[X \setminus C]^\vee) = 1$ if and only if
$\trans[X \setminus C]^\vee$ and hence
$\trans[X \setminus C]$ is trivial, which occurs if and only if
$X$ is a translation surface.
\end{proof}

\paragraph{Veech submersion again.}

As preparation for studying the isoresidual foliation, we give a
second proof of Theorem~\ref{T:Veechsubmersion} via analyzing the
exact sequence \eqref{eq:thissequenceisexact}.

\begin{proof}[Second proof of Theorem~\ref{T:Veechsubmersion}]
   Consider the  double complex of sheaves,
   \begin{equation}
     \label{eq:double_complex}
    \xymatrix{
      0 \ar[r] & 0 \ar[r] & \Omega_X(C) \ar[r]^{-\frac{1}{2\pi i}\cdot} & \Omega_X(C) \ar[r] & 0 \ar[r] & 0 \\
      0 \ar[r] & \trans \ar[r]\ar[u] & T_X(-C) \ar[r]^{^{\frac{1}{2 \pi i} \nabla^*}}\ar[u]^{\mathcal{L}_\nabla} & \mathcal{O}_X \ar[r]^{R_\Gamma}\ar[u]^d & W^*_{P_\bZ} \ar[r]\ar[u] & 0.
    }
  \end{equation}
  Replacing each column with the total complex 
  yields a new double complex which induces a spectral sequence where $E_{p,q}^1$ is the $q$th hypercohomology of the $p$th column:

  \begin{center}
    \begin{tikzpicture}
      \matrix (m) [matrix of math nodes,
      nodes in empty cells,nodes={minimum width=5ex,
        minimum height=5ex,outer sep=-5pt},
      column sep=1ex,row sep=1ex]{
        &      &     &     & &\\
        2     & H^2(\trans) & 0& 0& 0&\\
        1     & H^1(\trans)  &  \mathbb{H}^1(L^\bullet)  & H^1(X \setminus C) & 0 & \\
        0     &  0  & 0 &  H^0(X \setminus C) &   \Hom(W,\bC^{P_\bZ}) &  \\
        \quad\strut &   0  &  1  &  2  & 3 &\strut \\};
    \draw[thick] (m-1-1.east) -- (m-5-1.east) ;
    \draw[thick] (m-5-1.north) -- (m-5-5.north) ;
    \end{tikzpicture}
  \end{center}
  
  The differentials $d_1$ are the maps on hypercohomology induced by
  the horizontal arrows of \eqref{eq:double_complex}.  Since the rows
  of \eqref{eq:double_complex} are exact, this spectral sequence
  converges to $0$.  Our goal is to show that
  $d_1\colon E_{1,1}^1\to E_{2,1}^1$ is onto---equivalently
  $E_{2,1}^2=0$---if $X$ is not a translation surface of holomorphic
  type. If $X$ is not a translation surface then
  $h^2(\trans) = 0$ by Lemma \ref{L:H2Trans} so
  $E_{2,1}^2 = E_{2,1}^\infty = 0$.

  Suppose next that $X$ is a translation surface without poles. In
  that case, $E_{3,0}^1=0$, so for $E_{0,2}^\infty$ and
  $E_{2,1}^\infty$ to be $0$, we must have that
  $d_2\colon E_{0,2}^2=H^2(\trans)\to \mathrm{coker}(D\THol)$ is an
  isomorphism.  Therefore $D\THol$ has a one-dimensional cokernel by
  Lemma \ref{L:H2Trans}.

  Suppose finally that $X$ is a translation surface with poles.  The image of
  $d_1\colon H^0(X\setminus C)\to \Hom(W,\bC_{P_\bZ})$ is spanned the map
  $r\colon W\to \cx^{P_\zed}$, sending a flat one-form to its tuple of
  residues. By the Residue Theorem, this image is contained in $\Hom(W, H)$, where
  $H\subset \cx^{P_\zed}$ is the subspace of vectors whose coordinates
  sum to $0$.  In particular, $E_{3,0}^2$ is
  nonzero. We claim that $d_2 E_{1,1}^2$ is contained in the image of
  $\Hom(W, H)\to E_{3,0}^2$. If this is the case, then
  $d_3\colon E_{0,2}^3\to E_{3,0}^3$ must nonzero, so $E_{0,2}^3$ must
  be nonzero. This means that $d_2\colon E_{0,2}^2\to E_{2,1}^2$ must
  be $0$, since $E_{0,2}^2 = H^2(\trans)$ is one-dimensional. It
  follows that $E_{2,1}^2 = E_{2,1}^\infty=0$ as desired.

  The claim will follow from a geometric interpretation of
  $d_2\colon E_{1,1}^2\to E_{3,0}^2$ as a derivative of residues.  
  Consider a first order deformation
  $(\mathcal{X}\to B, \tnabla)$ of $(X, \nabla)$ in the kernel of
  $D\THol\colon\mathbb{H}^1(L^\bullet)\to H^1(X\setminus C)$.  As in the standard
  deformation theory setup, we take an open cover $U_\alpha$ together
  with nonzero holomorphic one-forms $\omega_\alpha$ on each
  $U_\alpha$ so that
  \begin{equation*}
    \tnabla \omega_\alpha =(\eta_\alpha+\epsilon\eta_\alpha')\otimes\omega_\alpha.  
  \end{equation*}
  The
  deformation is encoded by the cocycle $(\chi, \eta_\alpha')\in
  C^1(T_X(-C))\oplus C^0(\Omega_X)$.  The flat meromorphic one-form
  $\omega$ may then be written as $\omega=f_\alpha \omega_\alpha$ for
  some meromorphic functions $f_\alpha$.  (For this proof,
  $\omega$ should be a global meromorphic form, but for later
  use, we will also allow $(X, \nabla)$ to have nontrivial holonomy,
  in which case, $\omega$ will only be defined on a
  neighborhood of $\Gamma$, which may be taken to be a union of some
  of the $U_\alpha$.)

  Since $\omega$ is flat, we have on $U_\alpha\times B$,
  \begin{equation*}
    \tnabla \omega = df_\alpha\otimes\omega_\alpha +
    f_\alpha(\eta_\alpha + \epsilon \eta_\alpha')\otimes\omega_\alpha
    = \epsilon  \eta_\alpha'\otimes f_\alpha\omega_\alpha = \epsilon
    \eta_\alpha'\otimes \omega.
  \end{equation*}
  Since our deformation is in the kernel of $D\THol$.
  The cocycle
  $(\nabla\chi, -\eta')$ is trivial in $\mathbb{H}^1(\Omega_X^\bullet(
  C))$, so there is a cochain $\phi\in C^0(\mathcal{O}_X)$ such that
  \begin{equation}\label{E:d2Definition}
  \begin{aligned}
    \nabla^* \chi &= \delta \phi\\
    -\eta'&=d\phi. 
    \end{aligned}
  \end{equation}
  The differential $d_2$ is then given by 
  \[ d_2(\chi, \eta') = R_\Gamma(\phi) = (\omega\mapsto\mathrm{Res}_p
    \phi \omega)_{p\in P_\zed}.\] Fixing $c \in \bC$ and replacing
  each $\phi_\alpha$ with $\phi_\alpha + c$ yields another solution of
  Equation \ref{E:d2Definition}. Therefore, the value of
  $d_2(\chi, \eta')$ is only defined up to addition of a multiple of
  $r$.
    
  Note that if $p\in U_{\alpha\beta}$ is a pole, then
\[    \mathrm{Res}_p(\phi_\beta - \phi_\alpha)\omega = \mathrm{Res}_p(\nabla^*\chi_{\alpha\beta})\omega=    - \mathrm{Res}_p d(\iota_{\chi_{\alpha\beta}}\omega)=0,
\]  so this expression is well-defined.

  Define a meromorphic one-form $\tilde{\omega}_\alpha$ on
  $U_\alpha\times B$ by
  \begin{equation*}
    \tilde{\omega}_\alpha = (1+ \epsilon \phi_\alpha) \omega.
  \end{equation*}

  Since $(\phi_\beta - \phi_\alpha)\omega= (\nabla^*\chi)\omega =
  L_{\chi_{\alpha\beta}}\omega$ by \eqref{eq:liederivativeofomega}, the forms $\tilde{\omega}_\alpha$ glue to a global one-form
  $\tilde{\omega}$ on $\mathcal{X}$, and moreover,
  \begin{equation*}
    \tnabla \tilde{\omega} = \epsilon d\phi_\alpha\otimes \omega +
    (1+\epsilon\phi_\alpha)\epsilon\eta_\alpha'\otimes\omega =
    \epsilon(d\phi_\alpha+\eta_\alpha')\otimes\omega = 0,
  \end{equation*}
  so $\tilde{\omega}$ is flat.  We then have for each $p\in P_\zed$,
  \begin{equation*}
    \Res_p \tilde{\omega} = \Res_p\omega + \epsilon d_2(\chi, \eta')(\omega),
  \end{equation*}
  so $d_2(\chi, \eta')(\omega)$ can be interpreted as the derivatives of the
  residues of a flat deformation of $\omega$.  By the residue theorem, the sum of the coordinates of 
  $d_2(\chi, \eta')$ are zero.
\end{proof}

Consider now a leaf $\mathcal{L}\subset \TAres$ of the isoholonomic
locus through $X$, an affine surface which is not a translation
surface and has at least one integral pole of nonzero residue.  The
last argument of the previous proof allows us to express the derivative
$D(\Res_\Gamma|_\mathcal{L})$ in terms of the differential $d_2$.
More precisely, since $d_1$ is $D(\THol)$ on $E_{1,1}^1$, and
$T_X\mathcal{L}= \ker D(\THol)$ we have a quotient map $\pi\colon
T_X\mathcal{L}\to E_{1,1}^2$.  If we regard $W$ as embedded in
$\cx^{P_\zed}$ by the residue map $r\colon W\to \cx^{P_\zed}$, then
the $E_{3,0}^2$ term is the quotient $\Hom(W, \cx^{P_\zed})/ \Hom(W,
W) \isom \Hom(W, \cx^{P_\zed}/W)$.  This last space is naturally the
tangent space to $\proj^{P_\zed}$ at the projective class of $\Res_\Gamma(X)$.
In terms of these identifications, we have the following calculation
of $D(\Res_\Gamma|_\mathcal{L})$.

\begin{lem}
  \label{L:d2derivative}
  Suppose that $X$ is an affine surface which is not a translation
  surface and which has at least one integral pole of nonzero
  residue. Let $\cL$ be the leaf of the isoholonomic locus through
  $X$. Then the derivative at $X$ of the restriction of $\Res_\Gamma$ to
  $\mathcal{L}$ may be written in terms of the differential $d_2$ on
  the $E^2$ page of the above spectral sequence as
  \begin{equation*}
    D(\Res_\Gamma|_\mathcal{L}) = d_2\circ \pi.
  \end{equation*}
\end{lem}

\begin{proof}[Proof of Theorem~\ref{T:SubmersionWithResidue}]
  We first show that $D_X(\Hol\times \Res_\Gamma)$ is onto.  Since
  $D_X\Hol$ is onto, it suffices to establish that the restriction of
  $D_X \Res_\Gamma$ to $\ker( D_X\Hol)$ is onto.  By
  Lemma~\ref{L:d2derivative}, it suffices to show the surjectivity of 
  $d_2\colon E_{1,1}^2\to E_{3,0}^2$.   This map is in fact an isomorphism, since the spectral
  sequence converges to zero, and
  $H^2(\trans)=E_{0,2}^3$ by Lemma~\ref{L:H2Trans}.
  
  To see that the foliation defined by the fibers does not depend on the choice of the arcs
  $\Gamma$, suppose $\Gamma' = \{\gamma_2', \ldots, \gamma_k'\}$ is a
  second collection of embedded arcs joining the same points.  Let
  $\omega$ and $\omega'$ be flat one forms defined near $\Gamma$ and
  $\Gamma'$ respectively, chosen to agree near the pole $p_1$.  The
  residues of $\omega$ and $\omega'$ are related by
  \begin{equation*}
    \Res_{p_j} \omega' = \chi(\gamma_j^{-1}\gamma_j')\Res_{p_j} \omega'
  \end{equation*}
  where $\chi$ is the holonomy character of $X$.  It follows that
  $\Res_\Gamma = D\circ \Res_{\Gamma'}$, where $D\colon
  \cx^{P_\zed}\to \cx^{P_\zed}$ is a diagonal map, so the two maps
  have the same fibers.  
\end{proof}

Analyzing the same spectral sequence gives a description of the
tangent bundle of the isoresidual foliation.

\begin{prop}
  The tangent space $T_X \mathcal{L}$ to a leaf $\mathcal{L}$ of the isoresidual
  foliation through a surface $X$ in $\ModAres(\mu)$ is naturally
  identified with $H^1(\trans)$.
\end{prop}

\begin{proof}
  The tangent space to the fiber of $\Hol\times \Res_\Gamma$ through
  $X$ is the kernel of its derivative.  From
  Lemma~\ref{L:d2derivative} and the injectivity of $d_2\colon
  E_{1,1}^2\to E_{3,0}^2$, this is the same as the kernel of
  $\pi\colon T_X\mathcal{L}\to E_{1,1}^2$.  This kernel is the image
  of $H^1(\trans)\to \mathbb{H}^1(L^\bullet)$, and this map is
  injective because $E_{2,0}^2=0$ whenever $X$ has an integral pole of
  nonzero residue.
\end{proof}

By Proposition~\ref{P:TransExact}, we may identify $H^1(\trans)$ with
$H^1_c(X\setminus C'; \cx_\chi)$, where $C'=C\setminus P_\zed$ is the
set of cone points which are not integral poles, and $\cx_\chi$ is the
local coefficient system associated to the holonomy character $\chi$.

Let $\mathcal{F}'_{g,n}(\mu)\subset \ModAres(\mu)$ denote the locus of
flat surfaces, that is, affine surfaces with $S^1$-holonomy.  For such
surfaces, the cup product defines an indefinite Hermitian
form,
\begin{equation*}
  H^1_c(X\setminus C'; \cx_\chi)\otimes \overline{H^1_c(X\setminus C';
    \cx_\chi)}\to H^2_c(X\setminus C'; \cx) \isom \cx,
\end{equation*}
which is nondegenerate by Poincar\'e duality.  This defines a
leafwise indefinite metric for the isoresidual foliation of $\mathcal{F}'_{g,n}(\mu)$, analogous to
\cite{veech93}.